\documentclass[11pt]{article}

\usepackage{amssymb,amsfonts,amsmath}
\usepackage{enumerate}

\usepackage[english]{babel}
\usepackage[cp1250]{inputenc}

\usepackage{algorithm2e}

\usepackage{longtable}
\usepackage{changepage}

\usepackage{graphicx}
\usepackage{index}
\usepackage{fancyhdr}
\usepackage{lhelp}
\usepackage{optparams}
\usepackage{psfrag}
\usepackage{forloop}
\usepackage{setspace}
\usepackage{tikz}
\usepackage{subcaption}
\usepackage{pgffor}
\usepackage{xifthen}

\definecolor{lgray}{gray}{0.75}

\usetikzlibrary{decorations.pathreplacing,automata,calc,positioning}

\newcommand{\pch}{\chi_{\rho}}
\newcommand{\diam}{{\rm diam}}

\newcommand{\mptt}[1]{}

\newtheorem{theorem}{Theorem} 

\newtheorem{proposition}[theorem]{Proposition}

\newtheorem{remark}[theorem]{Remark}

\newcommand{\qed}{\hfill $\square$ \bigskip}

\begin{document}

\title{On packing total coloring}

\author{
Jasmina Ferme$^{a}$, Da\v sa Mesari\v c \v Stesl$^{b}$\\
jasmina.ferme1@um.si; Dasa.Stesl@fri.uni-lj.si
 }

\date{}

\maketitle

\begin{center}
$^a$ Faculty of Education, University of Maribor, Maribor, Slovenia\\
\medskip


$^b$ Faculty of Computer and Information Science, University of Ljubljana, Ljubljana, Slovenia\\
\medskip


\end{center}

\begin{abstract}
In this paper, we introduce a new concept in graph coloring, namely the \textit{packing total coloring}, which extends the idea of packing coloring to both the vertices and the edges of a given graph. More precisely, for a graph $G$, a packing total coloring is a mapping $c: V(G) \cup E(G) \rightarrow \{1, 2, \ldots\}$ with the property that for any integer $i$, any two distinct elements $A, B \in V(G) \cup E(G)$ with $c(A) = c(B) = i$ must be at distance at least $i+1$ from each other.  
Note that the distance between $A$ and $B$ means: a) the usual shortest-path distance between $A$ and $B$ if $A, B \in V(G)$; b) the $\min \{d(a,d), d(a,c),d(b,c), d(b,d)\}+1$ if $\{A, B\} =\{ab, cd\} \subseteq E(G)$; c) the $ \min \{d(a,X), d(b,X)\}+1$ if $\{A, B\}=\{ab, X\}$, where $ab \in E(G)$ and $X \in  V(G)$.
The smallest integer $k$ such that $G$ admits a packing total coloring using $k$ colors is called the \textit{packing total chromatic number}, denoted by $\chi_\rho^{''}(G)$.

In addition to introducing this new concept, we provide lower and upper bounds for the packing total chromatic numbers of graphs. Furthermore, we consider packing total chromatic numbers of graphs from the perspective of their maximum degrees and characterize all graphs $G$ with $\chi_\rho^{''}(G) \in \{1, 2, 3, 4, 5\}$.

\end{abstract}


\medskip\noindent
{\bf AMS Subj.\ Class: 05C15, 05C70, 05C69}

\section{Introduction}
Over the years, classical graph coloring has inspired the development of many new variants of colorings. One such variant is a packing coloring, introduced by Goddard, S.M.~Hedetniemi, S.T.~Hedetniemi, Harris and Rall in 2008~\cite{goddard-2008} under the name broadcast coloring, and defined as follows. Given a graph $G$ and a positive integer $k$, a \textit{$k$-packing coloring} is a function $c: V(G)\rightarrow \{1,2, \ldots, k\}$ such that for any distinct vertices $u,v \in V(G)$ and any $i \in \{1,2, \ldots, k\}$, the following implication holds: if $c(u)=c(v)=i$, then $d(u, v) > i$, where $d(u,v)$ denotes the usual shortest-path distance between $u$ and $v$. The {\em packing chromatic number} of $G$, denoted by $\pch(G)$, is the smallest integer $k$ such that there exists a $k$-packing coloring of $G$. As mentioned above, the concept of a packing coloring was initially named as broadcast coloring, which refers to its possible applications in the field of the assignment of broadcast frequencies to radio stations. Later, Bre\v sar, Klav\v zar and Rall~\cite{bkr-2007} renamed the term broadcast coloring to packing coloring since this concept is closely related to the concept of packing. 

The concept of packing coloring has attracted considerable attention from researchers, as evidenced by numerous publications on the topic, including a survey~\cite{survey} and several more recent studies (e.g.,~\cite{dliou, peterin, junosza, kristiana}).

A packing coloring is a special case of an $S$-packing coloring, defined as follows. If $S=(a_1,a_2,\ldots)$ is a non-decreasing sequence of positive integers, then an $S$-packing coloring of $G$ is a partition of $V(G)$ into sets $X_1,X_2,\ldots$ with the property that for each $i \in \{1,2,\ldots\}$, any two distinct vertices in $X_i$ are at distance greater than $a_i$. In the case when $a_i=i$ for each $i$, the notations and terminology reduces to packing coloring. 

While an $S$-packing coloring refers to a coloring of the vertices of a given graph, there also exists a natural variation focused on edge coloring, known as the \textit{$S$-packing edge coloring}, defined as follows.
If $S=(a_1,a_2,\ldots)$ is a non-decreasing sequence of positive integers, then an $S$-packing edge-coloring of $G$ is a partition of the edge set of $G$ into subsets
${E_1, E_2, \ldots}$ such that for any $i$ the distance between every pair of distinct edges $e_1, e_2 \in E_i$ is at least $a_i+1$. Note that the distance between $e_1=ab$ and $e_2=cd$ is $d(e_1,e_2)=\min \{d(a,d), d(a,c),d(b,c), d(b,d)\}+1$ which coincides with the distance between the corresponding vertices of $e_1$ and $e_2$ in the line graph of $G$. 
The concept of $S$-packing edge coloring is relatively new but has already been investigated in several papers (see, e.g.,~\cite{gt-2019, hocquard, liu, liu-santana, yang}). 
Similar to the case of $S$-packing coloring, a special case of $S$-packing edge-coloring, based on the sequence $(1,2,3, \ldots)$, is called a \textit{packing edge-coloring}. The smallest integer $k$ with the property that $G$ admits packing edge-coloring using $k$ colors is called \textit{packing chromatic index} of $G$ and is denoted by $\chi_\rho^{'}(G)$.

In this paper, we merge both the two aforementioned variations of packing colorings by introducing the concept of \textit{packing total coloring}. This new concept combines both packing coloring (of the vertices) and packing edge-coloring. Specifically, it refers to a coloring of both the vertices and edges of a graph, which is why it is termed \textit{total} coloring. The term \textit{packing} is used because the coloring must satisfy the conditions of a packing coloring on the vertex set and the conditions of a packing edge-coloring on the edge set.

A packing total coloring is a mapping $c: V(G) \cup E(G) \rightarrow \{1,2, \ldots\}$ such that for any $i \in \{1,2, \ldots\}$ and any two distinct elements $A, B \in V(G) \cup E(G)$ the following implication holds: if $c(A)=c(B)=i$, then $d(A, B) > i$, where $d(A,B)$ means a) the usual shortest-path distance between $A$ and $B$ if $A, B \in V(G)$; b) the $\min \{d(a,d), d(a,c),d(b,c), d(b,d)\}+1$ if $A, B \in E(G)$ and $A=ab$, $B=cd$; c) the $ \min \{d(a,B), d(b,B)\}+1$ if $A=ab \in E(G)$ and $B \in V(G)$; d) the value of $ \min \{d(A,c), d(A,d)\}+1$ if $A \in V(G)$ and $B=cd \in E(G)$. Therefore, the distance between $A, B \in V(G) \cup E(G)$ coincides with the distance between the corresponding vertices in the \textit{total graph} of a graph $G$, denoted by $T(G)$. Recall that $V(T(G))=V(G) \cup E(G)$ and $E(T(G))=E(G) \cup \{AB;~A,B \in E(G), A,B~\textrm{are incident in} ~G\} \cup  \{aB;~a \in V(G),B \in E(G), a~\textrm{is an endpoint of}~B~\textrm{in}~G\}$. Therefore, the concept of packing total coloring of a given graph $G$ is equivalent to a packing coloring (of the vertices) of the total graph of a graph $G$. 

In this paper, we will also use the term \textit{$k$-packing total coloring}, which refers to a packing total coloring that uses exactly $k$ colors. Next, the smallest integer $k$ such that there exists a $k$-packing total coloring of $G$, is called the \textit{packing total chromatic number} of $G$ and is denoted by $\chi_\rho^{''}(G)$. Note that this notation coincides with the notations for packing coloring (by using $\pch$) and total coloring (with $''$).  

In the continuation of this paper, we provide several lower and upper bounds for the packing total chromatic number of graphs. We also investigate the packing chromatic number in relation to the maximum degree of a graph. As we will see, the packing total chromatic number for any graph $G$ with $\Delta(G) \in \{1,2\}$ is at most $11$. In contrast, the packing total chromatic numbers within the family of cubic (or subcubic) graphs are unbounded. Furthermore, we characterize all graphs $G$ with $\chi_\rho^{''}(G) \in \{1,2,3,4, 5\}$. Finally, we present several questions and open problems for future research.


\section{Notations and preliminaries}

Let $G$ be a simple, finite and undirected graph. We denote by $V(G)$ the vertex set of $G$ and by $E(G)$ the edge set of $G$. 

Let $u$ and $v$ be arbitrary two vertices from $V(G)$. We say that $u$ is a \textit{neighbor} of $v$ ($v$ is a \emph{neighbor} of $u$, respectively) if $uv \in E(G)$. The \emph{(open) neighborhood} ($N_G(u)$) of $u$ is the set of all neighbors of $u$. Additionally, the cardinality of the set $N_G(u)$ is called the \emph{degree} of $u$ and is denoted by $deg_G(u)$. If $deg_G(u) = 1$, then $u$ is called a \textit{leaf}. Note that the subscript in the above notations may be omitted if the graph $G$ is clear from context.
Further, the \emph{maximum degree} of $G$, denoted by $\Delta(G)$, is defined as $\Delta(G) =\max\{deg(v);~v \in V(G)\}$ and the \emph{minimum degree} of $G$, denoted by $\delta(G)$ as $\delta(G) =\min\{deg(v);~v \in V(G)\}$. Recall that a \emph{cubic graph} is a graph in which every vertex has degree exactly three and a \emph{subcubic graph} is a graph with maximum degree at most three. An \emph{independent set} in a graph $G$ is a set of pairwise non-adjacent vertices and the \emph{independence number} of a graph $G$, denoted by $\alpha(G)$, is the cardinality of a maximum independent set in $G$.

Further, we say that edges $e_1, e_2\in E(G)$ are \emph{incident} if they have a common endpoint. The subset of edges $M\subseteq E(G)$ is called a \emph{matching} in a graph $G$ if no two edges of $M$ are incident. The \emph{matching number} of $G$, denoted by $\nu(G)$, is the cardinality of a maximum matching in $G$. 

A graph $H$ is a subgraph of $G$ if $V(H) \subseteq V(G)$ and $E(H) \subseteq E(G)$.

Next, a \emph{bipartite graph} is a graph $G=(V,E)$ whose vertex set can be partitioned into two disjoint subsets of vertices $V_1$ and $V_2$ such that there are no edges between vertices within the same subset $V_i$ for any $i\in\{1,2\}$. A bipartite graph $G$ is a \emph{complete bipartite graph} if any vertex from $V_1$ is adjacent to any vertex from $V_2$. A complete bipartite graph whose partite sets satisfy $|V_1|=m$ and $|V_2|=n$ is denoted by $K_{m,n}$. In addition, a \emph{star} is a complete bipartite graph $K_{1,n}$ consisting of a central vertex connected to $n$ leaves.

The distance between any two vertices $u,v \in V(G)$ is defined as the length of a shortest path between $u$ and $v$ in $G$. Next, $d(u,ab)=min\{d(u,a), d(u,b)\}+1$ for any $u \in V(G)$ and any $ab \in E(G)$. Finally, the distance between two edges in $G$ is defined as follows: for any $ab, cd \in E(G): d(ab,cd)=\min\{d(a,c), d(a,d), d(b,c), d(b,d)\}+1$.

For a given graph $G$, a \emph{line graph} of $G$, denoted by $L(G)$, is defined as follows: $V(L(G))=E(G)$ and two vertices from $L(G)$ are adjacent if and only if the corresponding edges from $G$ have a vertex in common. Further, given a graph $G$, the \emph{total graph} of $G$, denoted by $T(G)$, is the graph with the vertex set $
V(T(G)) = V(G) \cup E(G)$, in which two vertices are adjacent whenever they correspond to adjacent vertices in $G$, to incident edges in $G$, or when one corresponds to a vertex $v\in V(G)$ and the other to an edge $e\in E(G)$ with endpoint $v$.
Note that a total graph of a given graph $G$ can be also constructed from $G$ and $L(G)$ as follow. Let $V(T(G))=V(G) \cup V(L(G))$ and $E(T(G))=E(G) \cup E(L(G)) \cup \{(x,xy);~ x \in V(G) \land y \in V(G) \land xy \in E(G)\}$. In other words, we form $T(G)$ such that we first take one copy of $G$ and one copy of $L(G)$. Next, we connect each vertex from a copy of $L(G)$ with those two vertices from a copy of $G$ which represent the endpoints of an edge corresponding to the considered vertex in $L(G)$. More precisely, each vertex $xy$ from copy of $L(G)$ is adjacent to the vertices $x$ and $y$ from a copy of $G$ in $T(G)$. An example of a graph $T(G)$ formed from a given graph $G$ is shown in Fig.~\ref{fig:Example_L'(G)}.

In this paper, we introduce the concept of a packing total coloring. It is defined as a mapping $c: V(G) \cup E(G) \rightarrow \{1,2, \ldots\}$ with the property that for any $i \in \{1,2, \ldots\}$ and any two distinct elements $A, B \in V(G) \cup E(G)$ the following implication holds: if $c(A)=c(B)=i$, then the distance between $A$ and $B$ is at least $i+1$. Next, the smallest integer $k$ such that there exists a packing total coloring of $G$ using $k$ colors, is the packing total chromatic number of $G$ and is denoted by $\chi_\rho^{''}(G)$. In this paper, we also use the notation $|c^{-1}(i)|$ to denote the number of vertices in $G$ to which is assigned color $i$ by $c$.

Regarding the definition of a total graph of a given graph $G$ and using the fact that for any $A,B \in V(G) \cup E(G)$, $d_G(A,B)$ coincides with the distance between the vertices $A$ and $B$ in $T(G)$, we conclude that a concept of a packing total coloring of a given graph $G$ coincides with the concept of packing coloring (of the vertices) of $T(G)$. This means that for any graph $G$ holds the following: $\chi_\rho^{''}(G)=\chi_\rho(T(G))$.

\vspace{0.3cm}

\begin{figure}[ht]
\begin{center}
\begin{tikzpicture}[scale=0.75, style=thick]
\def\vr{3pt}
\def\len{1}

\coordinate(d) at (0.8,0);
\coordinate(e) at (1.6,-1);
\coordinate(a) at (0,2);
\coordinate(b) at (0.8,4);
\coordinate(c) at (0.8,3);
\draw (b)--(a)--(c);
\draw (a)--(d)--(e);
\draw(a)node[left]{$a$}; 
\draw(b)node[above]{$b$}; 
\draw(c)node[above]{$c$}; 
\draw(d)node[below]{$d$}; 
\draw(e)node[below]{$e$}; 
\coordinate(A) at (4.2,3.8);
\coordinate(B) at (5,1.3);
\coordinate(C) at (6,2.8);
\coordinate(D) at (5,0);
\draw (B)--(A)--(C)--(B)--(D);
\draw(A)node[above]{$ab$}; 
\draw(B)node[right]{$ad$}; 
\draw(C)node[above]{$ac$}; 
\draw(D)node[below]{$de$};  
%
\draw[thick, gray](c)--(C);
\draw[thick, gray](a)--(C);
\draw[thick, gray](a)--(B);
\draw[thick, gray](d)--(B);
\draw[thick, gray](a)--(A);
\draw[thick, gray](b)--(A);
\draw[thick, gray](d)--(D);
\draw[thick, gray](e)--(D);
%
\draw(a)[fill=white] circle(\vr);
\draw(b)[fill=white] circle(\vr);
\draw(c)[fill=white] circle(\vr);
\draw(d)[fill=white] circle(\vr);
\draw(e)[fill=white] circle(\vr);
\draw(A)[fill=white] circle(\vr);
\draw(B)[fill=white] circle(\vr);
\draw(C)[fill=white] circle(\vr);
\draw(D)[fill=white] circle(\vr);
\draw (0.9,1.4) ellipse (1.6cm and 3.5cm);
\draw(-0.5,4)node[above]{$G$}; 
\draw (4.9,1.4) ellipse (1.6cm and 3.5cm);
\draw(6.5,4)node[above]{$L(G)$}; 
\end{tikzpicture}
\end{center}
\caption{An example of $T(G)$.} 
\label{fig:Example_L'(G)}
\end{figure}
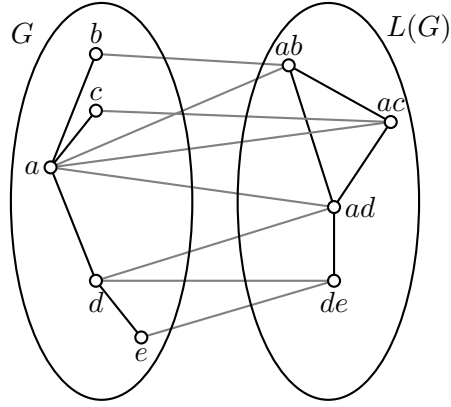




In the remainder of this section, we provide some general bounds for the packing total chromatic number of a graph.
\vspace{0.3cm}

A packing coloring assigns colors to the vertices of a given graph in such a way that any two distinct vertices colored with color $i$ are at distance greater than $i$. Since this condition (among others) must also be satisfied by a packing total coloring, we have the following proposition.

\begin{proposition}
    Let $G$ be a graph. Then, $$\chi_{\rho}^{''}(G)\geq \chi_{\rho}(G).$$
\end{proposition}

In addition, the difference between the packing total chromatic number and the packing chromatic number of a graph can be arbitrarily large. To prove that, consider a family of star graphs. We will show that for any positive integer $n$, $\chi_{\rho}^{''}(K_{1,n})=n+2$, while it is known that $\chi_{\rho}(K_{1,n})=2$. Consequently, for any $n$, $\chi_{\rho}^{''}(K_{1,n}) - \chi_{\rho}(K_{1,n})=n$. \\
First, we show that $\chi_{\rho}^{''}(K_{1,n}) \geq n+2$. Indeed, it is clear that any packing total coloring assigns to the edges of a star pairwise distinct colors, as all edges in this graph are pairwise incident. Additionally, the central vertex $x$ of the star needs to receive a color different from the colors of the edges, as it is the endpoint of each of these edges. Further, if all leaves of the star receive the same color, then this color is not assigned to $x$ neither to any  edge of a star. Hence, in this case, $\chi_{\rho}^{''}(K_{1,n}) \geq n+2$. Otherwise, at least one leaf receives a color which is not $1$. Therefore, this color is not assigned to $x$ neither to the edges of a star. Again, it follows that $\chi_{\rho}^{''}(K_{1,n}) \geq n+2$. \\
To show that $\chi_{\rho}^{''}(K_{1,n}) \leq n+2$, first color all leaves of $K_{1,n}$ with a color $1$. Then, color the vertex $x$ and all edges of the star with $n+1$ new, pairwise distinct colors. Since such coloring is a packing total coloring of $K_{1,n}$, it follows that $\chi_{\rho}^{''}(K_{1,n})=n+2$.\\


With the next proposition, we state that the packing total chromatic number has the hereditary property, which means that a graph cannot have a smaller packing total chromatic number than any of its subgraphs. Indeed, a packing total coloring of a given graph $G$, when restricted to the elements of $V(H) \cup E(H)$, where $H$ is an arbitrary subgraph of $G$, yields a packing total coloring of $H$, which clearly uses at most $\chi_{\rho}^{''}(G)$ colors.

\begin{proposition}
    For any graph $G$ and any subgraph $H$ of $G$, $\chi_{\rho}^{''}(H) \leq \chi_{\rho}^{''}(G)$.
    \label{hereditarnost}
\end{proposition}

Recall that $\chi_{\rho}^{''}(K_{1,n})=n+2$ holds for any $n$. Since each graph $G$ with at least one edge contains a subgraph isomorphic to $K_{1,\Delta(G)}$, Prop.~\ref{hereditarnost} implies the following lower bound for the packing total chromatic number of a graph, expressed in terms of its maximum degree.

\begin{proposition}
For any graph $G$ with at least one edge, $\chi_{\rho}^{''}(G) \geq \Delta(G)+2$. 
\label{lema_spodnja_meja}
\end{proposition}

%
%
%

Additionally, note that the above bound is sharp. Indeed, for a star graph, we have already showed that $\chi_{\rho}^{''}(K_{1,n}) = n + 2$, which equals $\Delta(K_{1,n}) + 2$. \\

Recall that Goddard, S.M.~Hedetniemi, S.T.~Hedetniemi, Harris and Rall~\cite{goddard-2008} proved the following upper bound for the packing chromatic number of a graph: $\chi_{\rho}(G) \leq |V(G)| - \alpha(G) + 1$. In addition, if $\mathrm{diam}(G) = 2$, then $\chi_{\rho}(G) = |V(G)| - \alpha(G) + 1$. These facts are mentioned here because they will be used multiple times in the remainder of this paper and also serve as the foundation for the following proposition.

\begin{proposition}
    For any connected graph $G$, $\chi_{\rho}^{''}(G)\leq |V(G)|+|E(G)|-\max\{\alpha(G), \nu(G)\}$\\$+1$. In addition, the bound is sharp.
    \label{zgornja_meja}
\end{proposition}

\begin{proof}
    Let $G$ be a connected graph. In order to prove the desired upper bound, we construct a packing total coloring $c$ of $G$ using $|V(G)|+|E(G)|-\max\{\alpha(G), \nu(G)\}+1$ colors. \\
    First, consider the case when $\max\{\alpha(G), \nu(G)\}=\alpha(G)$. Let $I$ be a maximum independent set of $G$, which means that $|I|=\alpha(G)$. Now, set $c(v)=1$ for any vertex $v\in I$ and let the vertices from $V(G)\setminus I$ and the edges from $E(G)$ receive pairwise distinct colors from the set of colors $\{2,3,\ldots ,|V(G)|+|E(G)|-\alpha(G)+1\}$. Obviously, the obtained coloring is a packing total coloring of $G$ using $|V(G)|+|E(G)|-\max\{\alpha(G), \nu(G)\}+1$ colors. \\
    It remains to consider the case when $\max\{\alpha(G), \nu(G)\}=\nu(G)$. Let $M$ be a matching of $G$ such that $|M|=\nu(G)$. Now, set $c(e)=1$ for any edge $e\in M$ and let the vertices from $V(G)$ and the edges of $E(G)\setminus M$ receive pairwise distinct colors from the set of colors $\{2,3,\ldots ,|V(G)|+|E(G)|-\nu(G)+1\}$. Again, the coloring $c$ is a packing total coloring of $G$ using $|V(G)|+|E(G)|-\max\{\alpha(G), \nu(G)\}+1$ colors. \\
    We conclude that $\chi_{\rho}^{''}(G)\leq |V(G)|+|E(G)|-\max\{\alpha(G), \nu(G)\}+1$. 

    To prove the tightness of the bound, consider an arbitrary star $K_{1,n}$.
    Since $\alpha(K_{1,n})=n$ and $\nu(K_{1,n})=1$, we have $|V(K_{1,n})|+|E(K_{1,n})|-\max\{\alpha(K_{1,n}), \nu(K_{1,n})\}+1=(n+1)+n-n+1=n+2= \chi_{\rho}^{''}(K_{1,n})$.
    This completes the proof.
    \qed
\end{proof}


\section{Packing total chromatic numbers of graphs in relation to their maximum degrees}

In this section, we consider the packing total chromatic numbers of connected graphs with respect to their maximum degrees. The only connected graph $G$ with $\Delta(G) = 1$ is the complete graph $K_2$, for which $\chi_{\rho}^{''}(K_2) = 3$. Any connected graph with $\Delta(G) = 2$ is isomorphic to either a cycle $C_n$ or a path $P_n$, where $n \geq 3$. As we show later in this section, the packing total chromatic number for graphs in the family $\{G;~\Delta(G) = 2\}$ is bounded above by $11$. In contrast, the packing total chromatic number in the family of (sub)cubic graphs is unbounded. This follows from the known fact that the packing chromatic number in this family is unbounded~\cite{balogh-2018}. Furthermore, an explicit infinite family of subcubic graphs with unbounded packing chromatic number constructed in~\cite{bf-2018b} also serves as an example of an infinite family of subcubic graphs with unbounded packing total chromatic number. Additionally, since any graph $G$ with $\Delta(G) > 3$ contains a subgraph isomorphic to a subcubic graph, it follows that the packing total chromatic number in this broader family is also unbounded.

In continuation, we consider packing total colorings of paths and cycles. Within the proofs, the distance graph $D(1,2)$ appears several times. Recall that this is the infinite graph with $\mathbb{Z}$ as its vertex set, where two distinct vertices $i, j \in \mathbb{Z}$ are adjacent if and only if $|i - j| \in \{1, 2\}$. 

Let $P_n$, $n \geq 3$, be a path and denote by $u_1, u_2, \ldots, u_n$ its consecutive vertices. Then $V(T(P_n))=\{u_1, u_2, \ldots, u_n, u_1u_2,u_2u_3, \ldots, u_{n-1}u_n\}$ and $E(T(P_n))=\{\{u_i,u_{i+1}\};~1 \leq i \leq n-1\} \cup \{\{u_i,u_iu_{i+1}\};~1 \leq i \leq n-1\} \cup \{\{u_i,u_{i-1}u_{i}\};~2 \leq i \leq n\}  \cup \{\{u_iu_{i+1}, u_{i+1}u_{i+2}\};$ $1 \leq i \leq n-2\}$. An example of a $T(P_n)$, namely $T(P_5)$ with vertices denoted as described, is shown in Fig.~\ref{fig:figure_T(P_5)}. Note that for each $n$, the total graph $T(P_n)$ is isomorphic to a subgraph of the distance graph $D(1,2)$.

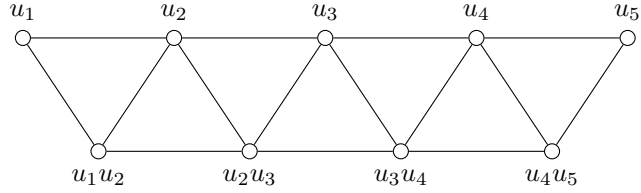
\begin{figure}[h]
\begin{center}
\begin{tikzpicture}[scale=1, 
  vertex/.style={circle, draw=black, fill=white, inner sep=2pt},
  every label/.style={font=\small}
  ]

  \node[vertex, label=above:$u_1$] (u1) at (0,0) {};
  \node[vertex, label=above:$u_2$] (u2) at (2,0) {};
  \node[vertex, label=above:$u_3$] (u3) at (4,0) {};
  \node[vertex, label=above:$u_4$] (u4) at (6,0) {};
  \node[vertex, label=above:$u_5$] (u5) at (8,0) {};

  \node[vertex, label=below:$u_1u_2$] (e1) at (1,-1.5) {};
  \node[vertex, label=below:$u_2u_3$] (e2) at (3,-1.5) {};
  \node[vertex, label=below:$u_3u_4$] (e3) at (5,-1.5) {};
  \node[vertex, label=below:$u_4u_5$] (e4) at (7,-1.5) {};

  \draw (u1) -- (u2);
  \draw (u2) -- (u3);
  \draw (u3) -- (u4);
  \draw (u4) -- (u5);

  \draw (u1) -- (e1);
  \draw (u2) -- (e1);
  \draw (u2) -- (e2);
  \draw (u3) -- (e2);
  \draw (u3) -- (e3);
  \draw (u4) -- (e3);
  \draw (u4) -- (e4);
  \draw (u5) -- (e4);

  \draw (e1)--(e2)--(e3)--(e4);
\end{tikzpicture}
\end{center}
\caption{A total graph of $P_5$ and notations of its vertices.}

\label{fig:figure_T(P_5)}
\end{figure}


With the following theorem we provide the exact values and bounds for the packing total chromatic numbers of paths $P_n$, $n \geq 3$.

\begin{theorem}
    $$\chi_{\rho}^{''}(P_n) =
\left\{\begin{array}{ll}
4; & n=3\,,\\
5; & n \in \{4,5\}\,,\\
6; & n \in \{6,7\}\,,\\
7; & n \in \{8,9, 10, 11, 12, 13\}\,.
\end{array}\right.$$
Moreover, for any $n \geq 14$, $\chi_{\rho}^{''}(P_n) \in \{7,8\}$.
\label{izrek:poti}
\end{theorem}

\begin{proof}
Let $P_n$, $n \geq 3$, be a path with the vertices denoted as described above.  

\vspace{0.3cm}
In the case when $n=3$, Prop.~\ref{lema_spodnja_meja} and Prop.~\ref{zgornja_meja} imply that $\chi_{\rho}^{''}(P_3) = 4$. 

\vspace{0.3cm}
Now, let $n=4$. Denote by $c$ any optimal packing coloring of $T(P_4)$. Suppose that $c$ uses at most $4$ colors. We observe that the vertices $u_2, u_3$ and $u_2u_3$ cannot receive a color $1$ by $c$. Indeed, if any of the mentioned vertices receive $1$, then its neighbors receive four distinct colors from $\{2,3,4\}$, which is not possible. Moreover, at least one of the vertices from $\{u_1, u_1u_2\}$ receive a color greater than $1$ by $c$. Since four of the vertices from $\{u_1, u_2, u_3, u_1u_2, u_2u_3\}$ receive colors from $\{2,3,4\}$ and any two of these vertices are pairwise at distance at most $2$, we derive that $c$ cannot be a packing coloring of $T(P_4)$. Hence, $\chi_{\rho}^{''}(P_4) \geq 5$. In order to prove that $\chi_{\rho}^{''}(P_4) = 5$, we form a $5$-packing coloring of $T(P_4)$ in the following way: let $c$ assigns a color $1$ to the vertices $u_1$, $u_4$ and $u_2u_3$ and four pairwise distinct colors to the remaining vertices. It follows that ${\chi_\rho}^{''}(P_4) = 5$.

\vspace{0.3cm}
Next, since $T(P_4)$ is a subgraph of $T(P_5)$, the hereditary property of packing coloring implies that $\chi_{\rho}^{''}(P_5) \geq 5$. Now, we form a $5$-packing total coloring $c$ of $P_5$ to show that $\chi_{\rho}^{''}(P_5) = 5$. Let $c(u_1)=c(u_4)=c(u_2u_3)=1$, $c(u_2)=c(u_4u_5)=2$, $c(u_5)=c(u_1u_2)=3$, $c(u_3)=4$ and $c(u_3u_4)=5$. Indeed, $c$ is a $5$-packing total coloring of $P_5$, hence our claim holds. 

\vspace{0.3cm}
For $P_6$ we want to prove that $\chi_{\rho}^{''}(P_6) = 6$. Suppose to the contrary that there exists an optimal $5$-packing coloring $c$ for $T(P_6)$. We observe that $|c^{-1}(1)| \leq 4$, $|c^{-1}(2)| \leq 3$, $|c^{-1}(3)| \leq 2$, $|c^{-1}(4)| \leq 2$, $|c^{-1}(5)| \leq 1$. Using these colors, $c$ can assign a color to at most $12$ vertices. In addition, note that $|V(T(P_6))|=11$. In continuation, we distinguish two cases regarding the number of the vertices of $T(P_6)$, which are colored with $2$ by $c$. \\
\textbf{Case A.} $|c^{-1}(2)| =3.$\\ 
To satisfy that any two vertices of $T(P_6)$, both colored with $2$, are at distance at least $3$, $c$ has to assign a color $2$ to the vertices $u_1, u_6$ and $u_3u_4$. Consequently, $|c^{-1}(4)| \leq 1$, which implies that all of the above written bounds should be achieved. Since $|c^{-1}(1)|=4$ and $|c^{-1}(3)|=2$, we have $\{c(u_2), c(u_5), c(u_1u_2), c(u_5u_6)\} \subseteq \{1,3\}$. If $c(u_1u_2)=c(u_5u_6)=3$, then $c(u_2)=c(u_5)=1$ and there are no additional vertices, which receive the color $1$ by $c$. Hence, requirement that $|c^{-1}(1)|=4$ is not achieved. Therefore, without loss of generality we can assume that $c(u_5)=c(u_1u_2)=3$. Then, $c$ assigns a color $1$ to the vertices $u_2$, $u_5u_6$ and $u_4$. Consequently, $|c^{-1}(1)|=3$, which implies that we can color only $10$ vertices by $c$, a contradiction since $|V(T(P_6))|=11$. \\
\textbf{Case B.} $|c^{-1}(2)| \leq 2.$ \\
Note that in this case, all of the above written bounds should be achieved, including $|c^{-1}(2)| =2$. Since $|c^{-1}(4)|=2$, one of the vertices from $\{u_1, u_1u_2\}$ and one of the vertices from $\{u_6, u_5u_6\}$ receive the color $4$ by $c$. Additionally, since $|c^{-1}(1)|=4$, at least one of the vertices from $\{u_1, u_1u_2, u_6, u_5u_6\}$ receives the color $1$. Without loss of generality we can assume that a vertex colored with $1$ belongs to $\{u_6, u_5u_6\}$. Next, since $|c^{-1}(3)|=2$, we derive that the vertices from $\{u_1, u_1u_2\}$ receive colors $3$ and $4$ by $c$. But then, $|c^{-1}(1)| \leq 3$, a contradiction. 

To prove that $\chi_{\rho}^{''}(P_6) = 6$, we construct a $6$-packing coloring $c$ of $T(P_6)$ as follows. Let $c(u_1)=c(u_3)=c(u_6)=c(u_4u_5)=1$, $c(u_2)=c(u_5)=2$, $c(u_1u_2)=c(u_5u_6)=3$, and let the remaining three vertices receive pairwise distinct colors from $\{4,5,6\}$. 

\vspace{0.3cm}
Now, let $n=7$. Since $T(P_7)$ contains a subgraph isomorphic to $T(P_6)$, we have $\chi_{\rho}^{''}(P_7) \geq 6$. Next, let $c$ assigns the color $1$ to the vertices $u_1, u_4, u_7, u_2u_3, u_5u_6$, the color $2$ to $u_2, u_5$, the color $3$ to $u_3, u_6u_7$, the color $4$ to $u_6, u_1u_2$, the color $5$ to $u_3u_4$, and $6$ to $u_4u_5$. Clearly, $c$ is a $6$-packing total coloring of $P_7$, hence $\chi_{\rho}^{''}(P_7) = 6$. 

\vspace{0.3cm}
For $P_8$ we will prove that $\chi_{\rho}^{''}(P_8) = 7$. First, suppose to the contrary that there exists a $6$-packing coloring $c$ of $T(P_8)$. We observe that $|c^{-1}(1)| \leq 5$, $|c^{-1}(2)| \leq 3$, $|c^{-1}(3)| \leq 3$, $|c^{-1}(4)|\leq2$, $|c^{-1}(5)| \leq 2$ and $|c^{-1}(6)| \leq 2$. In addition, note that $|V(T(P_8))|=15$.
In continuation of the proof we distinguish two cases regarding the number of the vertices of $T(P_8)$, which are colored with $3$ by $c$.  \\
\textbf{Case 1.} $|c^{-1}(3)|=3$. \\
In this case, $c(u_1)=c(u_8)=c(u_4u_5)=3$ and consequently $|c^{-1}(6)| = 1$. 
\begin{adjustwidth}{0.4cm}{0cm}
\textbf{Case 1.1} $|c^{-1}(1)| = 5$ and $|c^{-1}(5)| =2$.  \\
The conditions imply that $\{c(u_2), c(u_1u_2)\}=\{c(u_7), c(u_7u_8)\}=\{1,5\}$. Then, it follows that $|c^{-1}(2)| \leq 2$ and $|c^{-1}(4)| = 1$, which implies $\sum_{i=1}^6|c^{-1}(i)| \leq 14$, a contradiction to $c$ being a $6$-packing total coloring of $P_8$. \\
\textbf{Case 1.2} $|c^{-1}(1)| = 5$ and $|c^{-1}(5)| =1$. \\
In this case $\sum_{i=1}^6|c^{-1}(i)| = 15$, which means that all of the above written bounds must be achieved. If $c(u_3u_4)=2$ or $c(u_5u_6)=2$, then $|c^{-1}(2)| \leq 2$, a contradiction. Consequently, one of the vertices from $\{u_4, u_5\}$ receive $2$ by $c$.  Moreover, we derive that $\{c(u_4), c(u_5)\}=\{1,2\}$. Due to the symmetry of a graph and currently obtained packing coloring, we can assume that $c(u_4)=2$ and $c(u_5)=1$. But then it is easy to derive that $|c^{-1}(1)| \leq 4$, a contradiction to our assumption. \\
\textbf{Case 1.3} $|c^{-1}(1)| \leq 4$ and $|c^{-1}(5)| =2$. \\
We observe that at least one of the vertices of $\{u_2, u_1u_2\}$ and at least one of the vertices from $\{u_7, u_7u_8\}$ receive the color $5$ by $c$. Next, since $\sum_{i=1}^6|c^{-1}(i)| = 15$, $|c^{-1}(2)| =3$, which means that to at least one of the vertices from $\{u_2, u_1u_2, u_7, u_7u_8\}$ $c$ assigns the color $2$. Without loss of generality we may assume that $\{c(u_7u_8), c(u_7)\}=\{2,5\}$.
The fact that $|c^{-1}(1)| = 4$ implies $\{c(u_1u_2),c(u_2)\}=\{1,5\}$. But then there do not exist two vertices of $T(P_8)$ which can be colored with the color $4$, a contradiction. \\
\textbf{Case 1.4} $|c^{-1}(1)| \leq 4$ and $|c^{-1}(5)| =1$. \\
In this case $\sum_{i=1}^6|c^{-1}(i)| = 14$, which means that $c$ cannot assign colors to all vertices of $T(P_8)$. 
\vspace{0.2cm}
\end{adjustwidth} %
\textbf{Case 2.} $|c^{-1}(3)| \leq 2$. 
\begin{adjustwidth}{0.4cm}{0cm}
First, suppose that $|c^{-1}(1)| \leq 4$. This implies that $\sum_{i=1}^6|c^{-1}(i)| \leq 15$, which means that all of the above written bounds should be achieved, including $|c^{-1}(1)| = 4$. Therefore, $|c^{-1}(5)| = 2$ and $|c^{-1}(6)| = 2$, which implies that one of the vertices from $\{u_1, u_1u_2\}$ and one of the vertices from$\{u_8, u_7u_8\}$ receives the color $6$ by $c$. Additionally, we may assume that $c$ assigns $5$ to one of the vertices from $\{u_1, u_1u_2\}$ and to one vertex from $\{u_7, u_8, u_6u_7, u_7u_8\}$. If none of the vertices from $\{u_7, u_8, u_6u_7, u_7u_8\}$ receives $2$ by $c$, then $|c^{-1}(2)| = 2$, a contradiction. Therefore, at least one of the vertices from the mentioned set receives $2$ by $c$ and additionally, one of them receives $4$ by $c$. Consequently, $|c^{-1}(1)| = 3$, a contradiction to our assumption. Therefore, $|c^{-1}(1)| = 5$. In continuation of the proof, we distinguish four subcases.  

\textbf{Case 2.1} $|c^{-1}(5)| = 2$ and $|c^{-1}(6)| = 2$. \\
Our conditions yield that one of the vertices from $\{u_1, u_1u_2\}$ and one of the vertices from $\{u_8, u_7u_8\}$ receives the color $6$ by $c$. Next, at least one of the mentioned vertices, a vertex from $\{u_1, u_8, u_1u_2, u_7u_8\}$, receives the color $5$ by $c$. Without loss of generality we may assume that $c$ assigns $5$ to $u_1$ or $u_1u_2$. Therefore, $\{c(u_1),c(u_1u_2)\}=\{5,6\}$. Next, since $|c^{-1}(1)| = 5$, we derive that $c(u_8)=1$. Consequently, $c(u_6u_7)=c(u_5)=c(u_3u_4)=c(u_2)=1$ and $c(u_7)=5$.  But then $|c^{-1}(4)| =1$ and $|c^{-1}(2)| \leq 2$. Since $\sum_{i=1}^6|c^{-1}(i)| = 14$, we have a contradiction. 

\textbf{Case 2.2} $|c^{-1}(5)| = 2$ and $|c^{-1}(6)| = 1$. \\
Since $|V(T(P_8))|=15$ we infer that in this case $|c^{-1}(2)| = 3$, $|c^{-1}(3)|= |c^{-1}(4)|= 2$ and $|c^{-1}(1)| = 5$. Note that exactly one vertex from each of the vertex sets $\{u_1,u_2,u_1u_2\}$, $\{u_3,u_2u_3,u_3u_4\}$, $\{u_4,u_5,u_4u_5\}$, $\{u_6,u_5u_6,u_6u_7\}$ and $\{u_7,u_8,u_7u_8\}$ receives color $1$. Now, we distinguish two subcases.\\
\textbf{Subcase 2.2.1} $c(u_2)=1$. \\
Consequently, $c(u_3u_4)=c(u_5)=c(u_6u_7)=c(u_8)=1$. Further, since $|c^{-1}(2)|=3$ and $|c^{-1}(5)|=2$ we have $\{c(u_1),c(u_1u_2)\}=\{c(u_7),(u_7u_8)\}=\{2,5\}$. But then $|c^{-1}(4)|= 1$, a contradiction.\\
\textbf{Subcase 2.2.2} $c(u_1)=1$ or $c(u_1u_2)=1$. \\
Due to the symmetry of a graph $T(P_8)$ we can assume that $c(u_1)=1$. Additionally, we can also say that $c$ assigns the color $1$ to one of the vertices $u_8$ or $u_7u_8$. Namely, since $|c^{-1}(1)| = 5$ exactly one of the vertices $u_7$, $u_8$ and $u_7u_8$ receives color $1$ by $c$ and the case when $c(u_7)=1$ is analogues to the subcase 2.2.1 due to the symmetry of $T(P_8)$.\\
Further, since $|c^{-1}(5)| = 2$, $c$ assigns exactly one of the vertices $u_2$, $u_1u_2$, $u_2u_3$ the color $5$.\\
If $c(u_2u_3)=5$, then $c(u_8)=5$ and $c(u_7u_8)=1$. 
Moreover, $c$ assigns $1$ to five vertices, hence,  $c(u_6)=c(u_4u_5)=c(u_3)=1$. To satisfy the equalities $|c^{-1}(2)|=3$ and $|c^{-1}(4)|= 2$ we infer that $\{c(u_2),c(u_1u_2)\}=\{2,4\}$ and $\{c(u_7),c(u_6u_7)\}=\{2,4\}$. But then, $c$ can assign a color $3$ only to one vertex of $T(P_8)$, a contradiction. \\
%
Next, suppose that $c(u_2)=5$. It follows that $\{c(u_8),c(u_7u_8)\}=\{1,5\}$. Moreover, since $|c^{-1}(2)|= 3$, we know that $c(u_1u_2)=2$. Additionally, from the fact that $|c^{-1}(4)|= 2$ we derive that the vertices $u_2u_3$ and $u_7$ are colored with color $4$ by $c$. In order to satisfy the assumption $|c^{-1}(2)|= 3$ we set now $c(u_6u_7)=2$. But this means that $|c^{-1}(3)|\leq 1$, a contradiction.  \\
It remains to consider the case when $c(u_1u_2)=5$. Since $|c^{-1}(1)| = 5$, $|c^{-1}(2)| = 3$, $|c^{-1}(5)|=2$, $c$ assigns three distinct colors, namely $1$, $2$ and $5$, to the vertices $u_7$, $u_8$ and $u_7u_8$. Consequently,  $c(u_2)=c(u_6u_7)=4$ and then, $c(u_2u_3)=c(u_6)=3$. But this means that there are no five vertices, which are colored with $1$, a contradiction.  
%
\textbf{Case 2.3} $|c^{-1}(5)| = 1$ and $|c^{-1}(6)| = 2$. \\
Note that in this case $\sum_{i=1}^6|c^{-1}(i)| = 15$, which means that all of the above written bounds have to be achieved. Since $|c^{-1}(6)| = 2$ we know that one of the vertices from $\{u_1, u_1u_2\}$ and one of the vertices from $\{u_8, u_7u_8\}$ receives the color $6$ by $c$. Furthermore, at least one of the mentioned vertices receives $1$ by $c$. Withous loss of generality we may assume that $\{c(u_1), c(u_1u_2)\}=\{1,6\}$. \\
If $c(u_1)=6$ and $c(u_1u_2)=1$, then $\{c(u_8), c(u_7u_8)\}=\{1,6\}$. Moreover, since  $|c^{-1}(2)| = 3$ and  $|c^{-1}(4)| = 2$, we have $\{c(u_2),c(u_2u_3)\}=\{c(u_7), c(u_6u_7)\}=\{2,4\}$.  Consequently, it is impossible that two vertices from $T(P_8)$ receive a color $3$ by $c$, a contradiction. \\
Otherwise, if $c(u_1)=1$ and $c(u_1u_2)=6$, we know that $c(u_8)=6$ and one of the vertices from $\{u_7, u_7u_8\}$ receives $1$ by $c$. We have also found out that one of the vertices from $\{u_2, u_2u_3\}$ and one from $\{u_7, u_7u_8\}$ receive $2$ by $c$. Consequently, $c(u_2)=c(u_6u_7)=4$ and again, it is impossible to color two vertices from $T(P_8)$ by $3$, a contradiction.

\textbf{Case 2.4} $|c^{-1}(5)| = 1$ and $|c^{-1}(6)| = 1$. \\ 
Since $\sum_{i=1}^6|c^{-1}(i)| = 14$, we have a contradiction to $c$ being a $6$-packing total coloring of $P_8$, hence we are done. 
\end{adjustwidth}
\vspace{0.1cm}
Based on these findings, we can conclude that $\chi_{\rho}^{''}(P_8) \geq 7$. To prove that $\chi_{\rho}^{''}(P_8) \leq 7$, we form a $7$-packing coloring of $T(P_8)$ as follows. First, color the consecutive vertices $u_1, u_2, \ldots, u_8$ of $T(P_8)$ one after another by colors $1, 4,1,7,2,1,3,2$. Next, the consecutive vertices $u_1u_2, u_2u_3, \ldots, u_7u_8$ of $T(P_8)$ let receive colors $6,2,3,1,5,4,1$. Since the described coloring is indeed $7$-packing total coloring of $P_8$, we conclude that $\chi_{\rho}^{''}(P_8) = 7$.

\vspace{0.3cm}
Now, consider the case when $n=13$. Since $T(P_8)$ is a subgraph of $T(P_{13})$, we know that $\chi_{\rho}^{''}(P_{13}) \geq  7$. Further, color the vertices of $T(P_{13})$ as follows: let the vertices $u_1u_2, u_3, u_5, u_6u_7, u_8, u_9u_{10}, u_{11}u_{12}, u_{13}$ receive a color $1$, the vertices $u_1, u_3u_4, u_6, u_8u_9$ and $u_{11}$ a color $2$, to the vertices $u_2, u_5u_6, u_9, u_{12}u_{13}$ assign a color $3$, to $u_2u_3, u_7u_8, u_{12}$ a color $4$, to $u_4u_5$ and $u_{10}$ a color $5$, to $u_4$ and $u_{10}u_{11}$ a color $6$ and to the remaining vertex a color $7$. Clearly, this coloring is a $7$-packing total coloring of $P_{13}$, thus, $\chi_{\rho}^{''}(P_{13}) = 7$. Next, for any $i \in \{9, 10, 11, 12\}$, $P_8$ is a subgraph of $P_i$ and $P_i$ is a subgraph of $P_{13}$, hence we conclude that $\chi_{\rho}^{''}(P_{i}) =  7$.

\vspace{0.3cm}
Finally, let $n \geq 14$ be an arbitrary integer. The fact that $T(P_{13})$ is a subgraph of $T(P_n)$ implies $\chi_{\rho}^{''}(P_n) \geq  7$. On the other hand, note that $T(P_n)$ is a subgraph of a distance graph $D(1,2)$. Since $\chi_\rho(D(1,2))=8$ (see~\cite{togni-2014}), we derive that $\chi_{\rho}^{''}(P_n)$ is bounded from above by $8$. This concludes the proof.
    \qed
\end{proof}

 By Theorem \ref{izrek:poti}, we proved that for any $n \geq 14$, $\chi_{\rho}^{''}(P_n) \in \{7,8\}$. Moreover, using a computer, we found out that $\chi_{\rho}^{''}(P_n) =8$ for any $n \geq 14$. Note that it would be possible to provide a case-by-case proof for this result but that would take a very long time. Therefore, we have the following result. 
    
\begin{remark}
\label{remark:poti_racunalnik}
For any $n \geq 14$, $\chi_{\rho}^{''}(P_n) =8$.
\end{remark}

Next, we continue with determining the packing total chromatic numbers for cycles $C_n$, where $n \geq 3$.

Let $C_n$, $n \geq 3$, be a cycle. Denote the consecutive vertices of this cycle by $u_1, u_2, \ldots, u_n$. Then, $V(T(C_n))=\{u_1, u_2, \ldots, u_n, u_1u_2,u_2u_3, \ldots, u_{n-1}u_n, u_nu_1\}$ and $E(T(C_n))=\{\{u_i,u_{i+1}\};~1 \leq i \leq n-1\} \cup \{\{u_i,u_iu_{i+1}\};~1 \leq i \leq n-1\} \cup \{\{u_i,u_{i-1}u_{i}\};~2 \leq i \leq n\}  \cup \{\{u_iu_{i+1}, u_{i+1}u_{i+2}\};~1 \leq i \leq n-2\} \cup \{\{u_1,u_n\}, \{u_n, u_nu_1\},\{u_1, u_nu_1\}, \{u_{n-1}u_n, \\u_nu_1\}, \{u_1u_2, u_nu_1\}\}$. An example of a $T(C_n)$, namely $T(C_5)$, with the vertices denoted as described, is shown in Fig.~\ref{fig:figure_T(C_5)}.


\begin{figure}[h]
\begin{center}
\begin{tikzpicture}[scale=1.1, 
  vertex/.style={circle, draw=black, fill=white, inner sep=2pt},
  every label/.style={font=\small}
  ]
\node[vertex, label=above:$u_1$] (u1) at (90:2.5) {};
  \node[vertex, label=right:$u_2$] (u2) at (18:2.5) {};
  \node[vertex, label=below right:$u_3$] (u3) at (-54:2.5) {};
  \node[vertex, label=below left:$u_4$] (u4) at (-126:2.5) {};
  \node[vertex, label=left:$u_5$] (u5) at (162:2.5) {};

  \node[vertex, label=above left:$u_1u_2$] (e1) at (54:1.5) {};
  \node[vertex, label=left:$u_2u_3$] (e2) at (-18:1.5) {};
  \node[vertex, label=below:$u_3u_4$] (e3) at (-90:1.5) {};
  \node[vertex, label=right:$u_4u_5$] (e4) at (-162:1.5) {};
  \node[vertex, label=below right:$u_5u_1$] (e5) at (126:1.5) {};

  \draw (u1) -- (u2);
  \draw (u2) -- (u3);
  \draw (u3) -- (u4);
  \draw (u4) -- (u5);
  \draw (u5) -- (u1);

  \draw (u1) -- (e1);
  \draw (u2) -- (e1);
  \draw (u2) -- (e2);
  \draw (u3) -- (e2);
  \draw (u3) -- (e3);
  \draw (u4) -- (e3);
  \draw (u4) -- (e4);
  \draw (u5) -- (e4);
  \draw (u5) -- (e5);
  \draw (u1) -- (e5);

  \draw (e1) -- (e2);
  \draw (e2) -- (e3);
  \draw (e3) -- (e4);
  \draw (e4) -- (e5);
  \draw (e5) -- (e1);
\end{tikzpicture}
\end{center}
\caption{A total graph of $C_5$.}

\label{fig:figure_T(C_5)}
\end{figure}
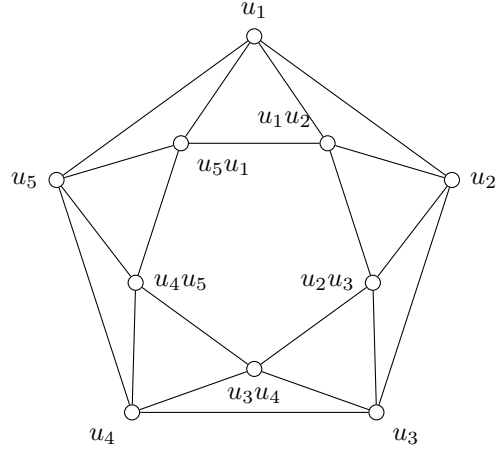


With the following theorem, we provide the exact values or bounds for the packing total chromatic numbers of cycles $C_n$, $n \geq 3$.

\begin{theorem}
    $$\chi_{\rho}^{''}(C_n) =
\left\{\begin{array}{ll}
5; & n=3\,,\\
7; & n \in \{4,5\}\,,\\
8; & n=6\,,\\
9; & n \in \{7,8,9,11\}\,,\\
10; & n \in \{10, 12, 13\}\,.
\end{array}\right.$$
Moreover, for any $n \geq 14$, $7  \leq \chi_{\rho}^{''}(C_n) \leq 11$.
\label{izrek:cikli}
\end{theorem}

\begin{proof}
Let $C_n$, $n \geq 3$, be a cycle with the vertices denoted as described above. 

\vspace{0.3cm}
First, let $n =3$. Since $\diam(T(C_3))=2$, any optimal packing coloring of $T(C_3)$ uses $|V(T(C_3)|-\alpha(T(C_3))+1$ colors. Using the fact that $\alpha(T(C_3))=2$, we derive that $\chi_{\rho}^{''}(C_3)=5$. Analogously we derive that $\chi_{\rho}^{''}(C_4)=7$ since $\diam(T(C_4))=2$ and $\alpha(T(C_4))=2$.

\vspace{0.3cm}
Now, consider a cycle $C_5$. We observe that $\diam(T(C_5))=3$ and there do not exist three vertices of $T(C_5)$ which are pairwise at distance $3$. Additionally, we can see that $\alpha(T(C_5))=3$. This means that applying any optimal packing coloring of $T(C_5)$ we can color at most three vertices with color $1$, at most two vertices with $2$ and only one vertex with any other color. Hence, $\chi_{\rho}^{''}(C_5) \geq 7$. Next, we form a packing coloring $c$ of $T(C_5)$ as follows: let $u_1, u_4$ and $u_2u_3$ receive color $1$ by $c$, the vertices $u_2$ and $u_4u_5$ a color $2$ by $c$ and the remaining vertices five additional colors. Since $c$ presents a $7$-packing total coloring of $C_5$, $\chi_{\rho}^{''}(C_5) = 7$.

\vspace{0.3cm}
We continue with considering a cycle $C_6$. Similarly as in the case of $C_5$, we observe that $\diam(T(C_6))=3$, in $T(C_6)$ do not exist three vertices which are pairwise at distance $3$ and $\alpha(T(C_6))=4$. Consequently, any optimal packing coloring of $T(C_6)$ assigns a color $1$ to at most four vertices, a color $2$ to at most two vertices and any other color to at most one vertex. Hence, $\chi_{\rho}^{''}(C_6) \geq 8$. In order to prove that $\chi_{\rho}^{''}(C_6) = 8$ we form a $8$-packing coloring $c$ of $T(C_6)$. Let $u_1, u_4, u_2u_3, u_5u_6$ receive $1$ by $c$, $u_2$ and $u_5$ a color $2$ by $c$ and the other vertices pairwise distinct colors from $\{3,4,5,6,7,8\}$. Indeed, $c$ is a $8$-packing total coloring of $C_6$ and thus, $\chi_{\rho}^{''}(C_6) = 8$.
\vspace{0.3cm}

Now, let $c$ be any optimal packing coloring of $T(C_7)$. We derive that $|c^{-1}(1)| \leq 5$, $|c^{-1}(2)| \leq 2$, $|c^{-1}(3)| \leq 2$ and $|c^{-1}(i)| \leq 1$ for any $i \geq 4$. If $\chi_{\rho}^{''}(C_7) = 8$ then all of these bounds should be achieved. But it is easy to see that there are no five vertices of $T(C_7)$ which all can receive a color $1$ by $c$. Thus,  $\chi_{\rho}^{''}(C_7) \geq 9$. Now, we form a $9$-packing coloring $c$ of $T(C_7)$. First, let $c(u_2)=c(u_4)=c(u_7)=c(u_5u_6)=1$, $c(u_3)=c(u_6u_7)=2$ and $c(u_6)=c(u_2u_3)=3$. Next, color the remaining six vertices with colors $4,5,6,7,8,9$. Therefore, $\chi_{\rho}^{''}(C_7) =9$.

\vspace{0.3cm}
Now, consider a cycle $C_8$. We observe that any optimal packing coloring of $T(C_8)$ assigns a color $1$ to at most five vertices, a color $2$ to at most three vertices, a color $3$ to at most two vertices and any other color to at most one vertex of $T(C_8)$. Consequently, $\chi_{\rho}^{''}(C_8) \geq 9$. Additionally, note that there exists an optimal packing coloring of $T(C_8)$ which assigns a color $1$ to the vertices $u_2, u_5, u_7, u_3u_4$ and $u_8u_1$, a color $2$ to the vertices $u_4, u_1u_2, u_6u_7$, a color $3$ to $u_1$ and $u_4u_5$, and six additional colors to the remaining six vertices. Consequently, $\chi_{\rho}^{''}(C_8) = 9$.

\vspace{0.3cm}
Similarly as in the case of $C_8$ we can also prove that $\chi_{\rho}^{''}(C_9) = 9$. Namely, any optimal packing coloring of $T(C_9)$ assigns a color $1$ to at most six vertices, a color $2$ to at most three vertices, colors $3$ and $4$ each to at most two vertices, and any other color to at most one vertex of $T(C_9)$. Since $|V(T(C_9))|=18$, we know that $\chi_{\rho}^{''}(C_9) \geq 9$. Further, we form a $9$-packing coloring $c$ of $T(C_9)$ to prove that $\chi_{\rho}^{''}(C_9) = 9$. First, let $c$ assigns a color $1$ to the vertices $u_2, u_5, u_8, u_3u_4, u_6u_7$ and $u_9u_1$, a color $2$ to the vertices $u_4, u_1u_2, u_7u_8$, a color $3$ to $u_7$ and $u_2u_3$, a color $4$ to $u_1$ and $u_5u_6$, and five distinct colors to the remaining five vertices. This coloring proves that $\chi_{\rho}^{''}(C_9) = 9$.

\vspace{0.3cm}
Next, let $n=10$. We observe that any packing coloring of $T(C_{10})$ assigns a color $1$ to at most six vertices, a color $2$ to at most four vertices, colors $3$ and $4$ each to at most two vertices and any other color to at most one vertex. Consequently, it uses at least $10$ colors, hence $\chi_{\rho}^{''}(C_{10}) \geq 10$. In order to prove that $\chi_{\rho}^{''}(C_{10}) \leq 10$, we construct the following packing coloring of $T(C_{10})$. Let the vertices $u_1u_2, u_3, u_5, u_6u_7, u_8u_9$ and $u_{10}$ receive a color $1$, the vertices $u_2, u_4u_5, u_7$ and $u_9u_{10}$ receive a color $2$, the vertices $u_4$ and $u_7u_8$ color with $3$, the vertices $u_3u_4$ and $u_8$ with $4$ and the remaining vertices with six additional colors. Then, $\chi_{\rho}^{''}(C_{10}) = 10$. 

\vspace{0.3cm}
Now, we consider the case of $C_{11}$. It is easy to see that any packing coloring of $T(C_{11})$ assigns a color $1$ to at most seven vertices, a color $2$ to at most four vertices, a color $3$ to at most three vertices, colors $4$ and $5$ each to at most two vertices and any other color to at most one vertex. Consequently, $\chi_{\rho}^{''}(C_{11}) \geq 9$. Now, we will prove that $\chi_{\rho}^{''}(C_{11}) \leq 9$ by constructing a $9$-packing coloring of $T(C_{11})$. Let the vertices $u_1, u_2u_3, u_4, u_5u_6, u_7, u_9$ and  $u_{10}u_{11}$ receive a color $1$, the vertices $u_2, u_4u_5, u_7u_8$ and $u_{10}$ receive a color $2$, the vertices $u_1u_2, u_5$ and $u_8u_9$ a color $3$, the vertices $u_6u_7$ and $u_{11}$ a color $4$, the vertices $u_6$ and $u_{11}u_{1}$ a color $5$ and the remaining vertices colors $6,7,8$ and $9$. Therefore, $\chi_{\rho}^{''}(C_{11}) = 9$. 

\vspace{0.3cm}
We continue with $C_{12}$. To prove that $\chi_{\rho}^{''}(C_{12}) \geq 10$ we suppose to the contrary that there exists a $9$-packing coloring $c$ of $T(C_{12})$. Note that $c$ assigns a color $1$ to at most eight vertices, a color $2$ to at most four vertices, a color $3$ to at most three vertices, colors $4$ and $5$ each to at most two vertices and any other color to at most one vertex. It follows that $c$ can assign colors only to $23$ vertices, but $|V(T(C_{12}))|=24$, a contradiction. To prove that $\chi_{\rho}^{''}(C_{12}) \leq 10$ we form a $10$-packing coloring of $T(C_{12})$. First, let the vertices $u_1, u_2u_3, u_4, u_5u_6, u_7, u_8u_9, u_{10}$ and $u_{11}u_{12}$ receive a color $1$, the vertices $u_2, u_5, u_8$ and $u_{10}u_{11}$ receive a color $2$, the vertices $u_3, u_6u_7$ and $u_{11}$ color with $3$, the vertices $u_6$ and $u_{12}$ with $4$, the vertices $u_1u_2$ and $u_{7}u_{8}$ with $5$ and the remaining vertices with five additional colors. Consequently, $\chi_{\rho}^{''}(C_{12}) = 10$. 

\vspace{0.3cm}
Next, consider a cycle $C_{13}.$ First, we prove that $\chi_{\rho}^{''}(C_{13}) \geq 10$. It is clear that any packing coloring $c$ of $T(C_{13})$ using $9$ colors assigns color $1$ to at most eight vertices, color $2$ to at most five vertices, color $3$ to at most three vertices, each of the colors $4, 5$ and $6$ to at most two vertices and each of the remaining three colors to at most one vertex. This implies that using $c$ we can color only $25$ vertices of $T(C_{13})$, whereas $|V(T(C_{13}))|=26$, a contradiction. Therefore, $\chi_{\rho}^{''}(C_{13}) \geq 10$. We continue by showing that $\chi_{\rho}^{''}(C_{13}) \leq 10$. To this end, we construct a $10$-packing coloring of $T(C_{13})$ as follows. Assign color $1$ to the vertices $u_1, u_2u_3, u_4, u_5u_6, u_7u_8, u_9$, $u_{10}u_{11}$ and $u_{12}u_{13}$, color $2$ to the vertices $u_1u_2, u_4u_5, u_7, u_9u_{10}$ and $u_{12}$, color $3$ to the vertices $u_2, u_6$ and $u_{10}$, color $4$ to the vertices $u_3$ and $u_8$, color $5$ to the vertices $u_3u_4$ and $u_{11}$, and color $6$ to the vertices $u_5$ and $u_{11}u_{12}$. Finally, to the remaining vertices assign colors $7,8,9$ and $10$. It follows that $\chi_{\rho}^{''}(C_{13}) = 10$.

\vspace{0.3cm}
Let $n \geq 14$ be an arbitrary integer. Using the facts that $\chi_\rho^{''}(P_n) \geq 7$ and $\chi_\rho(D(1,2))=8$~\cite{togni-2014}, we will prove that $7 \leq \chi_\rho^{''}(C_n) \leq 11$. 

Note that $C_n$ contains a subgraph, isomorphic to $P_n$, which means that $\chi_\rho^{''}(C_n) \geq \chi_\rho^{''}(P_n)$. By Theorem~\ref{izrek:poti}, $\chi_\rho^{''}(P_n) \geq 7$, which implies the desired lower bound for $\chi_\rho^{''}(C_n)$.

In order to prove that $\chi_\rho^{''}(C_n) \leq 11$, we first write $8$-packing coloring of $D(1,2)$, presented in~\cite{togni-2014}. This coloring will be used in the continuation of the proof and is constructed by repeating the following color pattern: $8, 1, 2, 6, 1, 4, 3, 2, 1, 5,7, 1, 2, 3, 4, 1, 6$, $2$, $1$, $8$, $3$, $1$, $2$, $4, 1, 5, 7, 1, 3, 2, 1, 6, 4, 1, 2, 3, 1, 8, 5, 1, 2, 4, 1, 3, 6, 1, 2, 7, 1, 5, 4, 2, 1, 3$.

To prove that $\chi_\rho^{''}(C_n) \leq 11$, we form a $11$-packing coloring $c$ of $T(C_n)$ using the presented packing coloring of $D(1,2)$. First, color the consecutive vertices $u_1, u_1u_2, u_2, u_2u_3$, $u_3, \ldots, u_n, u_nu_1$ of $T(C_n)$ one after another by applying the above presented color pattern (for $D(1,2)$). In this way, $c$ assigns colors to all vertices of $T(C_n)$, but in some cases it is not proper packing coloring of $T(C_n)$. Hence, in some cases, we have to re-color some vertices, namely to change the colors, which were assigned to these vertices by $c$. 
Since $T(P_n)$ is an induced subgraph of $D(1,2)$, $c$ is a proper $8$-packing coloring for $T(P_n)$. This implies that the distance between any two distinct vertices $u,v \in V(T(C_n))$ with $c(u)=c(v)=i$, $1 \leq i \leq 8$, is at least $i+1$, except maybe in the cases when $u$ (resp., $v$) belongs to $\{u_{n-8}, u_{n-8}u_{n-7}, u_{n-7}, \ldots, u_n, u_nu_1\}$ and $v$ (resp., $u$) belongs to $\{u_1, u_1u_2, u_2, \ldots, u_8, u_8u_9\}$. Considering the possibilities for currently obtained coloring $c$, we observe that inappropriate packing colorings of $T(C_n)$ occur if at least one of the following statements holds: 

\begin{itemize}
\item (A) $c(u_nu_1) \in \{1,5\}$;
\item (B) $c(u_n)=2$ or $c(u_nu_1)=2$;
\item (C) $c(u_n)=4$ or $c(u_nu_1)=4$ or $c(u_{n-1}u_n)=4$;
\item (D) a vertex from $\{u_{n-4}u_{n-3}, u_{n-3}, \ldots,u_n,u_nu_1\}$ receives a color $6$ by $c$;
\item (E) a vertex from $\{u_{n-1},u_{n-1}u_{n},u_n,u_nu_1\}$ receives a color $7$ by $c$;
\item (F) a vertex from $\{u_{n-7}, u_{n-7}u_{n-6}, \ldots,u_n,u_nu_1\}$ receives a color $8$ by $c$.
\end{itemize}
Note that, regardless of which vertex in $\{u_{n-8}, u_{n-8}u_{n-7}, u_{n-7}, \ldots, u_n, u_nu_1\}$ receives a color $3$ by $c$, any two distinct vertices of $T(C_n)$ that are both colored with $3$ are at a distance greater than $3$. Therefore, color $3$ does not cause any conflicts.

In continuation of this proof, we distinguish two cases based on the validity of statement (E).

\textbf{Case 1.} Statement (E) does not hold. 

An analysis of the presented color pattern for $D(1,2)$ and coloring $c$ shows that the conditions $c(u_n) = 2$ and $c(u_nu_1) \in \{1,5\}$ cannot hold at the same time. This means that $\chi_\rho^{''}(C_n) \leq 12$. \\
Next, if $V(T(C_n))$ contains two vertices at distance at most $2$, but both colored with $2$, then these two vertices are $u_2$ and one vertex from $\{u_n, u_nu_1\}$. \\
If at least one of the statements (C), (D) or (F) does not hold, then we can construct an $11$-packing coloring of $T(C_n)$ as follows. First, we replace a color $2$, assigned to a vertex from $\{u_n, u_nu_1\}$ with a color $9$. Next, we have to execute a recoloring of two of the following vertices: a vertex from $\{u_n, u_nu_1, u_{n-1}u_n\}$ colored with $4$, a vertex from $\{u_{n-4}u_{n-3}, u_{n-3}, \ldots,u_n,u_nu_1\}$ colored with $6$ and a vertex from $\{u_{n-7}, u_{n-7}u_{n-6}, \ldots,$ $u_n,u_nu_1\}$, which is colored by $8$. Indeed, since one of the statements (C), (D) or (F) does not hold, we have to recolor only two of the mentioned vertices with two additional colors, $10$ and $11$. In this way, we obtain an $11$-packing total coloring of $C_n$ and thus, $\chi_\rho^{''}(C_n) \leq 11$. \\
Now, consider the case when the statements (C), (D) and (F) hold, which means that $V(T(C_n))$ contains the vertices colored with colors $4, 6$ and $8$ that are at inappropriate distances. Recall that color $2$ is also problematic, since there exist two vertices of $V(T(C_n))$, both colored with $2$, which are at distance at most $2$.
More precisely, we have one of the situations presented in Fig.~\ref{fig:situation1}, Fig.~\ref{fig:situation2}
or Fig.~\ref{fig:situation3} (see the graph at the top of each figure). The problematic vertices and their colors are marked in red.
In any of these three situations, we can form a proper packing coloring of $T(C_n)$ using only $11$ colors. Indeed, such colorings are presented in Fig.~\ref{fig:situation1}, Fig.~\ref{fig:situation2}
and Fig.~\ref{fig:situation3} (see the graph presented at the bottom of each figure). Consequently, $\chi_\rho^{''}(C_n) \leq 11$. \\
It remains to consider the case when a color $2$ is not problematic (in other words, any two vertices of $T(C_n)$, both colored with $2$, are pairwise at distance at least $3$). \\
If also colors $1$ and $5$ are not problematic, then $\chi_\rho^{''}(C_n) \leq 11$. Indeed, in this case a vertex from $\{u_n, u_{n}u_1, u_{n-1}u_n\}$, colored with $4$, recolor using a color $9$. Next, a vertex from $\{u_{n-4}u_{n-3}, u_{n-3}, \ldots,u_n,u_nu_1\}$, colored with $6$, recolor by using a color $10$ and instead of a color $8$ on a vertex from $\{u_{n-7}, u_{n-7}u_{n-6}, \ldots,u_n,u_nu_1\}$ we use a color $11$. \\
Next, assume that a color $1$ or a color $5$ is problematic (recall that both cannot be problematic at the same time). Clearly, if at least one of the statements (C), (D) or (F) is not fulfilled, then we can form an $11$-packing coloring of $T(C_n)$ from $c$ by recoloring of three vertices from $V(T(C_n))$. Otherwise, we have a situation presented in Fig.~\ref{fig:situation4} (see the graph at the top of the figure). In this situation, we can construct a proper packing coloring of $T(C_n)$ as shown in Fig.~\ref{fig:situation4} - a graph at the bottom of the figure. Since this coloring uses only $11$ colors, $\chi_\rho^{''}(C_n) \leq 11$. 

\begin{figure}[h]
\begin{center}
%
\begin{tikzpicture}[scale=0.4, font=\small]
\def\dx{2} 
\def\dy{1.5} 
%
\foreach \i in {1,...,8} {
    \fill (\i*\dx, 0) circle (2pt);
  }
\node[left] at (1*\dx, 0) {$u_{1}$};
  \node[above] at (1*\dx, 0) {8};
  \node[above] at (2*\dx, 0) {2};
  \node[above] at (3*\dx, 0) {1};
  \node[above] at (4*\dx, 0) {3};
  \node[above] at (5*\dx, 0) {1};
  \node[above] at (6*\dx, 0) {7};
  \node[above] at (7*\dx, 0) {2};
  \node[above] at (8*\dx, 0) {4};

\foreach \i in {1,...,7} {
    \pgfmathtruncatemacro{\j}{\i + 1}
    \coordinate (mid\i) at ({(\i)*\dx}, -\dy);
    \fill (mid\i) circle (2pt);

    \draw (mid\i) -- (\i*\dx, 0);
    \draw (mid\i) -- (\j*\dx, 0);
    \draw (\i*\dx, 0) -- (\j*\dx, 0);}
    \coordinate (mid8) at ({8*\dx}, -\dy);
    \fill (mid8) circle (2pt);
      \draw (mid8) -- (8*\dx, 0);

\node[below] at (1*\dx, -\dy) {1};
\node[below] at (2*\dx, -\dy) {6};
\node[below] at (3*\dx, -\dy) {4};
\node[below] at (4*\dx, -\dy) {2};
\node[below] at (5*\dx, -\dy) {5};
\node[below] at (6*\dx, -\dy) {1};
\node[below] at (7*\dx, -\dy) {3};
\node[below] at (8*\dx, -\dy) {1};

\foreach \i in {1,...,7} {
    \pgfmathtruncatemacro{\j}{\i + 1}
    \draw (mid\i) -- (mid\j);}
\node at ({(9)*\dx}, 0) {$\cdots$};
\node at ({(9)*\dx}, -\dy) {$\cdots$};
\node[above] at (10*\dx, 0) {1};
\node[above] at (11*\dx, 0) {7};
\node[above] at (12*\dx, 0) {2};
\node[above] at (13*\dx, 0) {4};
\node[above] at (14*\dx, 0) {\textcolor{red}{6}};
\node[above] at (15*\dx, 0) {1};
\node[above] at (16*\dx, 0) {3};
\node[above] at (17*\dx, 0) {\textcolor{red}{2}};

\node[below] at (10*\dx, -\dy) {5};
\node[below] at (11*\dx, -\dy) {1};
\node[below] at (12*\dx, -\dy) {3};
\node[below] at (13*\dx, -\dy) {1};
\node[below] at (14*\dx, -\dy) {2};
\node[below] at (15*\dx, -\dy) {\textcolor{red}{8}};
\node[below] at (16*\dx, -\dy) {1};
\node[below] at (17*\dx, -\dy) {\textcolor{red}{4}};

\foreach \i in {10,...,17} {
 \pgfmathtruncatemacro{\b}{17-\i}
    \coordinate (U\i) at ({(\i)*\dx}, 0);
    \fill (\i*\dx, 0) circle (2pt);
    }


\foreach \i in {10,...,16} {
    \pgfmathtruncatemacro{\j}{\i + 1}
    \coordinate (mid\i) at ({(\i)*\dx}, -\dy);
    \fill (mid\i) circle (2pt);
   \draw (mid\i) -- (\i*\dx, 0);
   \draw (mid\i) -- (\j*\dx, 0);
    \draw (\i*\dx, 0) -- (\j*\dx, 0);
    \draw (\i*\dx, -\dy) -- (\j*\dx, -\dy);}
\draw (17*\dx,0) -- (17*\dx,-\dy);
\draw (\dx,0) -- (17*\dx,-\dy);
\draw[black, thick]
  ([shift={(0,0)}] 0:2 and 2) 
  arc[start angle=180, end angle=0, x radius=8*\dx, y radius=2];
\draw[black, thick]
  ([shift={(16*\dx,-\dy)}] 0:2 and 2) 
  arc[start angle=0, end angle=-180, x radius=8*\dx, y radius=2];
\end{tikzpicture}

$\downarrow$ \vspace{0.1cm}

\begin{tikzpicture}[scale=0.4, font=\small]
\def\dx{2} 
\def\dy{1.5} 
\foreach \i in {1,...,8} {
    \fill (\i*\dx, 0) circle (2pt);}

\node[left] at (1*\dx, 0) {$u_{1}$};
  \node[above] at (1*\dx, 0) {8};
  \node[above] at (2*\dx, 0) {2};
  \node[above] at (3*\dx, 0) {1};
  \node[above] at (4*\dx, 0) {3};
  \node[above] at (5*\dx, 0) {1};
  \node[above] at (6*\dx, 0) {7};
  \node[above] at (7*\dx, 0) {2};
  \node[above] at (8*\dx, 0) {4};
  
\foreach \i in {1,...,7} {
    \pgfmathtruncatemacro{\j}{\i + 1}
    \coordinate (mid\i) at ({(\i)*\dx}, -\dy);
    \fill (mid\i) circle (2pt);
    \draw (mid\i) -- (\i*\dx, 0);
    \draw (mid\i) -- (\j*\dx, 0);
    \draw (\i*\dx, 0) -- (\j*\dx, 0);
}
    \coordinate (mid8) at ({8*\dx}, -\dy);
    \fill (mid8) circle (2pt);
      \draw (mid8) -- (8*\dx, 0);

      \node[below] at (1*\dx, -\dy) {1};
\node[below] at (2*\dx, -\dy) {6};
\node[below] at (3*\dx, -\dy) {4};
\node[below] at (4*\dx, -\dy) {2};
\node[below] at (5*\dx, -\dy) {5};
\node[below] at (6*\dx, -\dy) {1};
\node[below] at (7*\dx, -\dy) {3};
\node[below] at (8*\dx, -\dy) {1};

\foreach \i in {1,...,7} {
    \pgfmathtruncatemacro{\j}{\i + 1}
    \draw (mid\i) -- (mid\j);}
\node at ({(9)*\dx}, 0) {$\cdots$};
\node at ({(9)*\dx}, -\dy) {$\cdots$};
\node[above] at (10*\dx, 0) {1};
\node[above] at (11*\dx, 0) {7};
\node[above] at (12*\dx, 0) {2};
\node[above] at (13*\dx, 0) {4};
\node[above] at (14*\dx, 0) {\textcolor{green}{9}};
\node[above] at (15*\dx, 0) {1};
\node[above] at (16*\dx, 0) {3};
\node[above] at (17*\dx, 0) {\textcolor{green}{5}};

\node[below] at (10*\dx, -\dy) {5};
\node[below] at (11*\dx, -\dy) {1};
\node[below] at (12*\dx, -\dy) {3};
\node[below] at (13*\dx, -\dy) {1};
\node[below] at (14*\dx, -\dy) {2};
\node[below] at (15*\dx, -\dy) {\textcolor{green}{10}};
\node[below] at (16*\dx, -\dy) {1};
\node[below] at (17*\dx, -\dy) {\textcolor{green}{11}};

\foreach \i in {10,...,17} {
 \pgfmathtruncatemacro{\b}{17-\i}
    \coordinate (U\i) at ({(\i)*\dx}, 0);
    \fill (\i*\dx, 0) circle (2pt);
    }
    
\foreach \i in {10,...,16} {
    \pgfmathtruncatemacro{\j}{\i + 1}
    \coordinate (mid\i) at ({(\i)*\dx}, -\dy);
    \fill (mid\i) circle (2pt);
   \draw (mid\i) -- (\i*\dx, 0);
   \draw (mid\i) -- (\j*\dx, 0);
    \draw (\i*\dx, 0) -- (\j*\dx, 0);
    \draw (\i*\dx, -\dy) -- (\j*\dx, -\dy);}
\draw (17*\dx,0) -- (17*\dx,-\dy);
\draw (\dx,0) -- (17*\dx,-\dy);
\draw[black, thick]
  ([shift={(0,0)}] 0:2 and 2) 
  arc[start angle=180, end angle=0, x radius=8*\dx, y radius=2];
\draw[black, thick]
  ([shift={(16*\dx,-\dy)}] 0:2 and 2) 
  arc[start angle=0, end angle=-180, x radius=8*\dx, y radius=2];
\end{tikzpicture}
\end{center}
\caption{Improper and proper packing coloring of $T(C_n)$ - situation 1.}
\label{fig:situation1}
\end{figure}
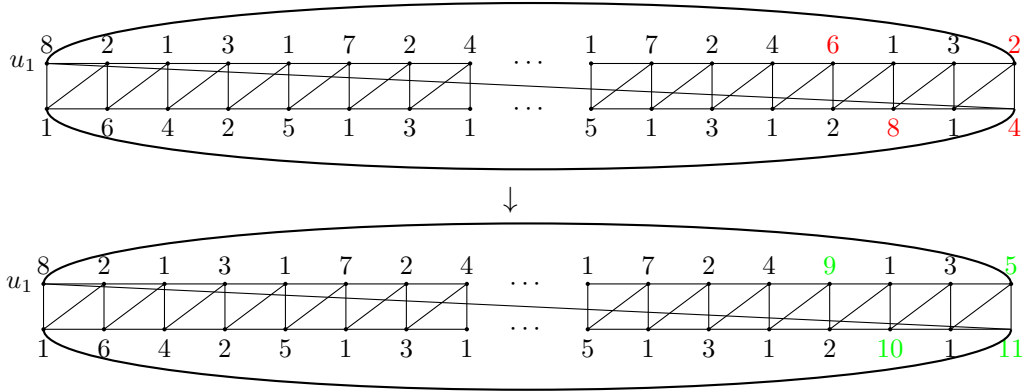
\begin{figure}[h]
\begin{center}

\begin{tikzpicture}[scale=0.4, font=\small]
\def\dx{2} 
\def\dy{1.5} 

\foreach \i in {1,...,8} {
    \fill (\i*\dx, 0) circle (2pt);
  }
\node[left] at (1*\dx, 0) {$u_{1}$};
  \node[above] at (1*\dx, 0) {8};
  \node[above] at (2*\dx, 0) {2};
  \node[above] at (3*\dx, 0) {1};
  \node[above] at (4*\dx, 0) {3};
  \node[above] at (5*\dx, 0) {1};
  \node[above] at (6*\dx, 0) {7};
  \node[above] at (7*\dx, 0) {2};
  \node[above] at (8*\dx, 0) {4};

\foreach \i in {1,...,7} {
    \pgfmathtruncatemacro{\j}{\i + 1}
    \coordinate (mid\i) at ({(\i)*\dx}, -\dy);
    \fill (mid\i) circle (2pt);

    \draw (mid\i) -- (\i*\dx, 0);
    \draw (mid\i) -- (\j*\dx, 0);
    \draw (\i*\dx, 0) -- (\j*\dx, 0);}
    \coordinate (mid8) at ({8*\dx}, -\dy);
    \fill (mid8) circle (2pt);
      \draw (mid8) -- (8*\dx, 0);

\node[below] at (1*\dx, -\dy) {1};
\node[below] at (2*\dx, -\dy) {6};
\node[below] at (3*\dx, -\dy) {4};
\node[below] at (4*\dx, -\dy) {2};
\node[below] at (5*\dx, -\dy) {5};
\node[below] at (6*\dx, -\dy) {1};
\node[below] at (7*\dx, -\dy) {3};
\node[below] at (8*\dx, -\dy) {1};

\foreach \i in {1,...,7} {
    \pgfmathtruncatemacro{\j}{\i + 1}
    \draw (mid\i) -- (mid\j);}
\node at ({(9)*\dx}, 0) {$\cdots$};
\node at ({(9)*\dx}, -\dy) {$\cdots$};
\node[above] at (10*\dx, 0) {1};
\node[above] at (11*\dx, 0) {5};
\node[above] at (12*\dx, 0) {2};
\node[above] at (13*\dx, 0) {1};
\node[above] at (14*\dx, 0) {\textcolor{red}{6}};
\node[above] at (15*\dx, 0) {2};
\node[above] at (16*\dx, 0) {1};
\node[above] at (17*\dx, 0) {\textcolor{red}{4}};

\node[below] at (10*\dx, -\dy) {\textcolor{red}{8}};
\node[below] at (11*\dx, -\dy) {1};
\node[below] at (12*\dx, -\dy) {4};
\node[below] at (13*\dx, -\dy) {3};
\node[below] at (14*\dx, -\dy) {1};
\node[below] at (15*\dx, -\dy) {7};
\node[below] at (16*\dx, -\dy) {5};
\node[below] at (17*\dx, -\dy) {\textcolor{red}{2}};

\foreach \i in {10,...,17} {
 \pgfmathtruncatemacro{\b}{17-\i}
    \coordinate (U\i) at ({(\i)*\dx}, 0);
    \fill (\i*\dx, 0) circle (2pt);
    }


\foreach \i in {10,...,16} {
    \pgfmathtruncatemacro{\j}{\i + 1}
    \coordinate (mid\i) at ({(\i)*\dx}, -\dy);
    \fill (mid\i) circle (2pt);
   \draw (mid\i) -- (\i*\dx, 0);
   \draw (mid\i) -- (\j*\dx, 0);
    \draw (\i*\dx, 0) -- (\j*\dx, 0);
    \draw (\i*\dx, -\dy) -- (\j*\dx, -\dy);}
\draw (17*\dx,0) -- (17*\dx,-\dy);
\draw (\dx,0) -- (17*\dx,-\dy);
\draw[black, thick]
  ([shift={(0,0)}] 0:2 and 2) 
  arc[start angle=180, end angle=0, x radius=8*\dx, y radius=2];
\draw[black, thick]
  ([shift={(16*\dx,-\dy)}] 0:2 and 2) 
  arc[start angle=0, end angle=-180, x radius=8*\dx, y radius=2];
\end{tikzpicture}

$\downarrow$ \vspace{0.1cm}

\begin{tikzpicture}[scale=0.4, font=\small]
\def\dx{2} 
\def\dy{1.5} 
\foreach \i in {1,...,8} {
    \fill (\i*\dx, 0) circle (2pt);}

\node[left] at (1*\dx, 0) {$u_{1}$};
  \node[above] at (1*\dx, 0) {\textcolor{green}{2}};
  \node[above] at (2*\dx, 0) {\textcolor{green}{8}};
  \node[above] at (3*\dx, 0) {1};
  \node[above] at (4*\dx, 0) {3};
  \node[above] at (5*\dx, 0) {1};
  \node[above] at (6*\dx, 0) {7};
  \node[above] at (7*\dx, 0) {2};
  \node[above] at (8*\dx, 0) {4};
  
\foreach \i in {1,...,7} {
    \pgfmathtruncatemacro{\j}{\i + 1}
    \coordinate (mid\i) at ({(\i)*\dx}, -\dy);
    \fill (mid\i) circle (2pt);
    \draw (mid\i) -- (\i*\dx, 0);
    \draw (mid\i) -- (\j*\dx, 0);
    \draw (\i*\dx, 0) -- (\j*\dx, 0);
}
    \coordinate (mid8) at ({8*\dx}, -\dy);
    \fill (mid8) circle (2pt);
      \draw (mid8) -- (8*\dx, 0);

      \node[below] at (1*\dx, -\dy) {1};
\node[below] at (2*\dx, -\dy) {6};
\node[below] at (3*\dx, -\dy) {4};
\node[below] at (4*\dx, -\dy) {2};
\node[below] at (5*\dx, -\dy) {5};
\node[below] at (6*\dx, -\dy) {1};
\node[below] at (7*\dx, -\dy) {3};
\node[below] at (8*\dx, -\dy) {1};

\foreach \i in {1,...,7} {
    \pgfmathtruncatemacro{\j}{\i + 1}
    \draw (mid\i) -- (mid\j);}
\node at ({(9)*\dx}, 0) {$\cdots$};
\node at ({(9)*\dx}, -\dy) {$\cdots$};
\node[above] at (10*\dx, 0) {1};
\node[above] at (11*\dx, 0) {5};
\node[above] at (12*\dx, 0) {2};
\node[above] at (13*\dx, 0) {1};
\node[above] at (14*\dx, 0) {\textcolor{green}{9}};
\node[above] at (15*\dx, 0) {2};
\node[above] at (16*\dx, 0) {1};
\node[above] at (17*\dx, 0) {\textcolor{green}{10}};

\node[below] at (10*\dx, -\dy) {8};
\node[below] at (11*\dx, -\dy) {1};
\node[below] at (12*\dx, -\dy) {4};
\node[below] at (13*\dx, -\dy) {3};
\node[below] at (14*\dx, -\dy) {1};
\node[below] at (15*\dx, -\dy) {7};
\node[below] at (16*\dx, -\dy) {5};
\node[below] at (17*\dx, -\dy) {\textcolor{green}{11}};

\foreach \i in {10,...,17} {
 \pgfmathtruncatemacro{\b}{17-\i}
    \coordinate (U\i) at ({(\i)*\dx}, 0);
    \fill (\i*\dx, 0) circle (2pt);
    }
    
\foreach \i in {10,...,16} {
    \pgfmathtruncatemacro{\j}{\i + 1}
    \coordinate (mid\i) at ({(\i)*\dx}, -\dy);
    \fill (mid\i) circle (2pt);
   \draw (mid\i) -- (\i*\dx, 0);
   \draw (mid\i) -- (\j*\dx, 0);
    \draw (\i*\dx, 0) -- (\j*\dx, 0);
    \draw (\i*\dx, -\dy) -- (\j*\dx, -\dy);}
\draw (17*\dx,0) -- (17*\dx,-\dy);
\draw (\dx,0) -- (17*\dx,-\dy);
\draw[black, thick]
  ([shift={(0,0)}] 0:2 and 2) 
  arc[start angle=180, end angle=0, x radius=8*\dx, y radius=2];
\draw[black, thick]
  ([shift={(16*\dx,-\dy)}] 0:2 and 2) 
  arc[start angle=0, end angle=-180, x radius=8*\dx, y radius=2];
\end{tikzpicture}
\end{center}
\caption{Improper and proper packing coloring of $T(C_n)$ - situation 2.}
\label{fig:situation2}
\end{figure}
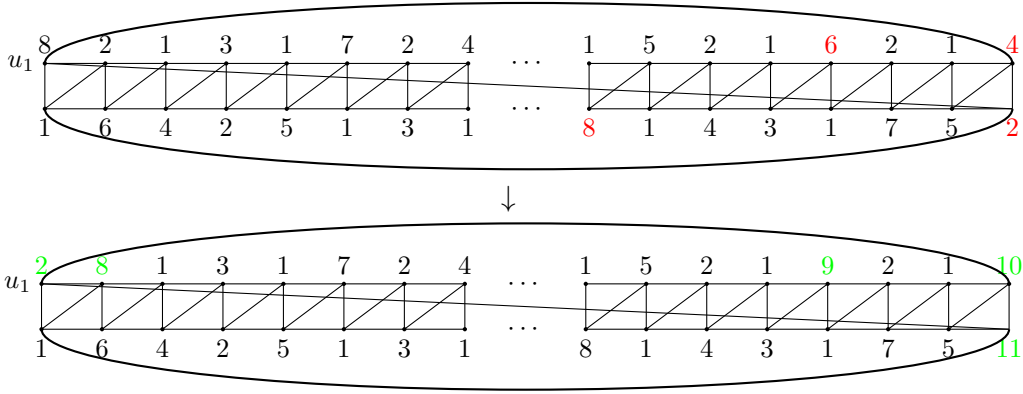
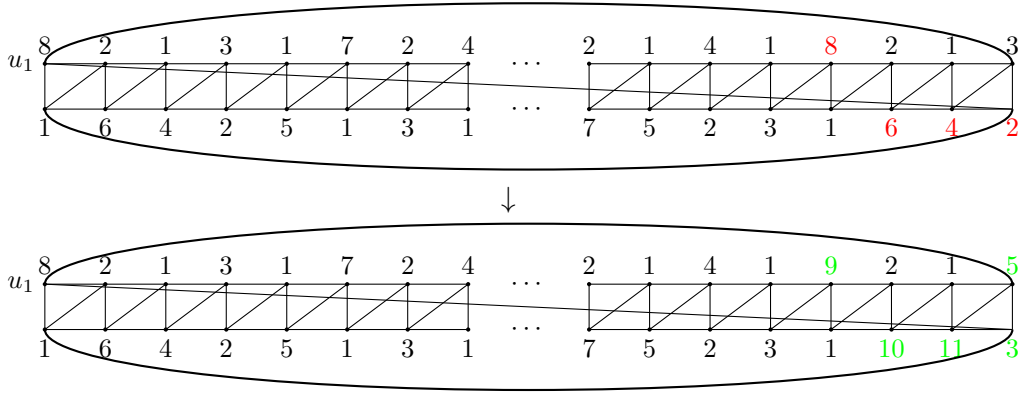
\begin{figure}[h]
\begin{center}
\begin{tikzpicture}[scale=0.4, font=\small]
\def\dx{2} 
\def\dy{1.5} 

\foreach \i in {1,...,8} {
    \fill (\i*\dx, 0) circle (2pt);
  }
\node[left] at (1*\dx, 0) {$u_{1}$};
  \node[above] at (1*\dx, 0) {8};
  \node[above] at (2*\dx, 0) {2};
  \node[above] at (3*\dx, 0) {1};
  \node[above] at (4*\dx, 0) {3};
  \node[above] at (5*\dx, 0) {1};
  \node[above] at (6*\dx, 0) {7};
  \node[above] at (7*\dx, 0) {2};
  \node[above] at (8*\dx, 0) {4};

\foreach \i in {1,...,7} {
    \pgfmathtruncatemacro{\j}{\i + 1}
    \coordinate (mid\i) at ({(\i)*\dx}, -\dy);
    \fill (mid\i) circle (2pt);

    \draw (mid\i) -- (\i*\dx, 0);
    \draw (mid\i) -- (\j*\dx, 0);
    \draw (\i*\dx, 0) -- (\j*\dx, 0);}
    \coordinate (mid8) at ({8*\dx}, -\dy);
    \fill (mid8) circle (2pt);
      \draw (mid8) -- (8*\dx, 0);

\node[below] at (1*\dx, -\dy) {1};
\node[below] at (2*\dx, -\dy) {6};
\node[below] at (3*\dx, -\dy) {4};
\node[below] at (4*\dx, -\dy) {2};
\node[below] at (5*\dx, -\dy) {5};
\node[below] at (6*\dx, -\dy) {1};
\node[below] at (7*\dx, -\dy) {3};
\node[below] at (8*\dx, -\dy) {1};

\foreach \i in {1,...,7} {
    \pgfmathtruncatemacro{\j}{\i + 1}
    \draw (mid\i) -- (mid\j);}
\node at ({(9)*\dx}, 0) {$\cdots$};
\node at ({(9)*\dx}, -\dy) {$\cdots$};
\node[above] at (10*\dx, 0) {2};
\node[above] at (11*\dx, 0) {1};
\node[above] at (12*\dx, 0) {4};
\node[above] at (13*\dx, 0) {1};
\node[above] at (14*\dx, 0) {\textcolor{red}{8}};
\node[above] at (15*\dx, 0) {2};
\node[above] at (16*\dx, 0) {1};
\node[above] at (17*\dx, 0) {3};

\node[below] at (10*\dx, -\dy) {7};
\node[below] at (11*\dx, -\dy) {5};
\node[below] at (12*\dx, -\dy) {2};
\node[below] at (13*\dx, -\dy) {3};
\node[below] at (14*\dx, -\dy) {1};
\node[below] at (15*\dx, -\dy) {\textcolor{red}{6}};
\node[below] at (16*\dx, -\dy) {\textcolor{red}{4}};
\node[below] at (17*\dx, -\dy) {\textcolor{red}{2}};

\foreach \i in {10,...,17} {
 \pgfmathtruncatemacro{\b}{17-\i}
    \coordinate (U\i) at ({(\i)*\dx}, 0);
    \fill (\i*\dx, 0) circle (2pt);
    }


\foreach \i in {10,...,16} {
    \pgfmathtruncatemacro{\j}{\i + 1}
    \coordinate (mid\i) at ({(\i)*\dx}, -\dy);
    \fill (mid\i) circle (2pt);
   \draw (mid\i) -- (\i*\dx, 0);
   \draw (mid\i) -- (\j*\dx, 0);
    \draw (\i*\dx, 0) -- (\j*\dx, 0);
    \draw (\i*\dx, -\dy) -- (\j*\dx, -\dy);}
\draw (17*\dx,0) -- (17*\dx,-\dy);
\draw (\dx,0) -- (17*\dx,-\dy);
\draw[black, thick]
  ([shift={(0,0)}] 0:2 and 2) 
  arc[start angle=180, end angle=0, x radius=8*\dx, y radius=2];
\draw[black, thick]
  ([shift={(16*\dx,-\dy)}] 0:2 and 2) 
  arc[start angle=0, end angle=-180, x radius=8*\dx, y radius=2];
\end{tikzpicture}

$\downarrow$ \vspace{0.1cm}

\begin{tikzpicture}[scale=0.4, font=\small]
\def\dx{2} 
\def\dy{1.5} 
\foreach \i in {1,...,8} {
    \fill (\i*\dx, 0) circle (2pt);}

\node[left] at (1*\dx, 0) {$u_{1}$};
  \node[above] at (1*\dx, 0) {8};
  \node[above] at (2*\dx, 0) {2};
  \node[above] at (3*\dx, 0) {1};
  \node[above] at (4*\dx, 0) {3};
  \node[above] at (5*\dx, 0) {1};
  \node[above] at (6*\dx, 0) {7};
  \node[above] at (7*\dx, 0) {2};
  \node[above] at (8*\dx, 0) {4};
  
\foreach \i in {1,...,7} {
    \pgfmathtruncatemacro{\j}{\i + 1}
    \coordinate (mid\i) at ({(\i)*\dx}, -\dy);
    \fill (mid\i) circle (2pt);
    \draw (mid\i) -- (\i*\dx, 0);
    \draw (mid\i) -- (\j*\dx, 0);
    \draw (\i*\dx, 0) -- (\j*\dx, 0);
}
    \coordinate (mid8) at ({8*\dx}, -\dy);
    \fill (mid8) circle (2pt);
      \draw (mid8) -- (8*\dx, 0);

      \node[below] at (1*\dx, -\dy) {1};
\node[below] at (2*\dx, -\dy) {6};
\node[below] at (3*\dx, -\dy) {4};
\node[below] at (4*\dx, -\dy) {2};
\node[below] at (5*\dx, -\dy) {5};
\node[below] at (6*\dx, -\dy) {1};
\node[below] at (7*\dx, -\dy) {3};
\node[below] at (8*\dx, -\dy) {1};

\foreach \i in {1,...,7} {
    \pgfmathtruncatemacro{\j}{\i + 1}
    \draw (mid\i) -- (mid\j);}
\node at ({(9)*\dx}, 0) {$\cdots$};
\node at ({(9)*\dx}, -\dy) {$\cdots$};
\node[above] at (10*\dx, 0) {2};
\node[above] at (11*\dx, 0) {1};
\node[above] at (12*\dx, 0) {4};
\node[above] at (13*\dx, 0) {1};
\node[above] at (14*\dx, 0) {\textcolor{green}{9}};
\node[above] at (15*\dx, 0) {2};
\node[above] at (16*\dx, 0) {1};
\node[above] at (17*\dx, 0) {\textcolor{green}{5}};

\node[below] at (10*\dx, -\dy) {7};
\node[below] at (11*\dx, -\dy) {5};
\node[below] at (12*\dx, -\dy) {2};
\node[below] at (13*\dx, -\dy) {3};
\node[below] at (14*\dx, -\dy) {1};
\node[below] at (15*\dx, -\dy) {\textcolor{green}{10}};
\node[below] at (16*\dx, -\dy) {\textcolor{green}{11}};
\node[below] at (17*\dx, -\dy) {\textcolor{green}{3}};

\foreach \i in {10,...,17} {
 \pgfmathtruncatemacro{\b}{17-\i}
    \coordinate (U\i) at ({(\i)*\dx}, 0);
    \fill (\i*\dx, 0) circle (2pt);
    }
    
\foreach \i in {10,...,16} {
    \pgfmathtruncatemacro{\j}{\i + 1}
    \coordinate (mid\i) at ({(\i)*\dx}, -\dy);
    \fill (mid\i) circle (2pt);
   \draw (mid\i) -- (\i*\dx, 0);
   \draw (mid\i) -- (\j*\dx, 0);
    \draw (\i*\dx, 0) -- (\j*\dx, 0);
    \draw (\i*\dx, -\dy) -- (\j*\dx, -\dy);}
\draw (17*\dx,0) -- (17*\dx,-\dy);
\draw (\dx,0) -- (17*\dx,-\dy);
\draw[black, thick]
  ([shift={(0,0)}] 0:2 and 2) 
  arc[start angle=180, end angle=0, x radius=8*\dx, y radius=2];
\draw[black, thick]
  ([shift={(16*\dx,-\dy)}] 0:2 and 2) 
  arc[start angle=0, end angle=-180, x radius=8*\dx, y radius=2];
\end{tikzpicture}
\end{center}
\caption{Improper and proper packing coloring of $T(C_n)$ - situation 3.}
\label{fig:situation3}
\end{figure}
\begin{figure}[h]
\begin{center}
\begin{tikzpicture}[scale=0.4, font=\small]
\def\dx{2} 
\def\dy{1.5} 

\foreach \i in {1,...,8} {
    \fill (\i*\dx, 0) circle (2pt);
  }
\node[left] at (1*\dx, 0) {$u_{1}$};
  \node[above] at (1*\dx, 0) {8};
  \node[above] at (2*\dx, 0) {2};
  \node[above] at (3*\dx, 0) {1};
  \node[above] at (4*\dx, 0) {3};
  \node[above] at (5*\dx, 0) {1};
  \node[above] at (6*\dx, 0) {7};
  \node[above] at (7*\dx, 0) {2};
  \node[above] at (8*\dx, 0) {4};

\foreach \i in {1,...,7} {
    \pgfmathtruncatemacro{\j}{\i + 1}
    \coordinate (mid\i) at ({(\i)*\dx}, -\dy);
    \fill (mid\i) circle (2pt);

    \draw (mid\i) -- (\i*\dx, 0);
    \draw (mid\i) -- (\j*\dx, 0);
    \draw (\i*\dx, 0) -- (\j*\dx, 0);}
    \coordinate (mid8) at ({8*\dx}, -\dy);
    \fill (mid8) circle (2pt);
      \draw (mid8) -- (8*\dx, 0);

\node[below] at (1*\dx, -\dy) {1};
\node[below] at (2*\dx, -\dy) {6};
\node[below] at (3*\dx, -\dy) {4};
\node[below] at (4*\dx, -\dy) {2};
\node[below] at (5*\dx, -\dy) {5};
\node[below] at (6*\dx, -\dy) {1};
\node[below] at (7*\dx, -\dy) {3};
\node[below] at (8*\dx, -\dy) {1};
%
%
\foreach \i in {1,...,7} {
    \pgfmathtruncatemacro{\j}{\i + 1}
    \draw (mid\i) -- (mid\j);}
\node at ({(9)*\dx}, 0) {$\cdots$};
\node at ({(9)*\dx}, -\dy) {$\cdots$};
\node[above] at (10*\dx, 0) {1};
\node[above] at (11*\dx, 0) {3};
\node[above] at (12*\dx, 0) {2};
\node[above] at (13*\dx, 0) {1};
\node[above] at (14*\dx, 0) {7};
\node[above] at (15*\dx, 0) {3};
\node[above] at (16*\dx, 0) {1};
\node[above] at (17*\dx, 0) {\textcolor{red}{4}};

\node[below] at (10*\dx, -\dy) {\textcolor{red}{8}};
\node[below] at (11*\dx, -\dy) {1};
\node[below] at (12*\dx, -\dy) {4};
\node[below] at (13*\dx, -\dy) {5};
\node[below] at (14*\dx, -\dy) {1};
\node[below] at (15*\dx, -\dy) {2};
\node[below] at (16*\dx, -\dy) {\textcolor{red}{6}};
\node[below] at (17*\dx, -\dy) {\textcolor{red}{1}};

\foreach \i in {10,...,17} {
 \pgfmathtruncatemacro{\b}{17-\i}
    \coordinate (U\i) at ({(\i)*\dx}, 0);
    \fill (\i*\dx, 0) circle (2pt);
    }

%
%
\foreach \i in {10,...,16} {
    \pgfmathtruncatemacro{\j}{\i + 1}
    \coordinate (mid\i) at ({(\i)*\dx}, -\dy);
    \fill (mid\i) circle (2pt);
   \draw (mid\i) -- (\i*\dx, 0);
   \draw (mid\i) -- (\j*\dx, 0);
    \draw (\i*\dx, 0) -- (\j*\dx, 0);
    \draw (\i*\dx, -\dy) -- (\j*\dx, -\dy);}
\draw (17*\dx,0) -- (17*\dx,-\dy);
\draw (\dx,0) -- (17*\dx,-\dy);
\draw[black, thick]
  ([shift={(0,0)}] 0:2 and 2) 
  arc[start angle=180, end angle=0, x radius=8*\dx, y radius=2];
\draw[black, thick]
  ([shift={(16*\dx,-\dy)}] 0:2 and 2) 
  arc[start angle=0, end angle=-180, x radius=8*\dx, y radius=2];
\end{tikzpicture}

$\downarrow$ \vspace{0.1cm}

\begin{tikzpicture}[scale=0.4, font=\small]
\def\dx{2} 
\def\dy{1.5} 
\foreach \i in {1,...,8} {
    \fill (\i*\dx, 0) circle (2pt);}

\node[left] at (1*\dx, 0) {$u_{1}$};
  \node[above] at (1*\dx, 0) {\textcolor{green}{2}};
  \node[above] at (2*\dx, 0) {\textcolor{green}{8}};
  \node[above] at (3*\dx, 0) {1};
  \node[above] at (4*\dx, 0) {3};
  \node[above] at (5*\dx, 0) {1};
  \node[above] at (6*\dx, 0) {7};
  \node[above] at (7*\dx, 0) {2};
  \node[above] at (8*\dx, 0) {4};
  
\foreach \i in {1,...,7} {
    \pgfmathtruncatemacro{\j}{\i + 1}
    \coordinate (mid\i) at ({(\i)*\dx}, -\dy);
    \fill (mid\i) circle (2pt);
    \draw (mid\i) -- (\i*\dx, 0);
    \draw (mid\i) -- (\j*\dx, 0);
    \draw (\i*\dx, 0) -- (\j*\dx, 0);
}
    \coordinate (mid8) at ({8*\dx}, -\dy);
    \fill (mid8) circle (2pt);
      \draw (mid8) -- (8*\dx, 0);

      \node[below] at (1*\dx, -\dy) {1};
\node[below] at (2*\dx, -\dy) {6};
\node[below] at (3*\dx, -\dy) {4};
\node[below] at (4*\dx, -\dy) {2};
\node[below] at (5*\dx, -\dy) {5};
\node[below] at (6*\dx, -\dy) {1};
\node[below] at (7*\dx, -\dy) {3};
\node[below] at (8*\dx, -\dy) {1};

\foreach \i in {1,...,7} {
    \pgfmathtruncatemacro{\j}{\i + 1}
    \draw (mid\i) -- (mid\j);}
\node at ({(9)*\dx}, 0) {$\cdots$};
\node at ({(9)*\dx}, -\dy) {$\cdots$};
\node[above] at (10*\dx, 0) {1};
\node[above] at (11*\dx, 0) {3};
\node[above] at (12*\dx, 0) {2};
\node[above] at (13*\dx, 0) {1};
\node[above] at (14*\dx, 0) {7};
\node[above] at (15*\dx, 0) {3};
\node[above] at (16*\dx, 0) {1};
\node[above] at (17*\dx, 0) {\textcolor{green}{9}};

\node[below] at (10*\dx, -\dy) {8};
\node[below] at (11*\dx, -\dy) {1};
\node[below] at (12*\dx, -\dy) {4};
\node[below] at (13*\dx, -\dy) {5};
\node[below] at (14*\dx, -\dy) {1};
\node[below] at (15*\dx, -\dy) {2};
\node[below] at (16*\dx, -\dy) {\textcolor{green}{10}};
\node[below] at (17*\dx, -\dy) {\textcolor{green}{11}};

\foreach \i in {10,...,17} {
 \pgfmathtruncatemacro{\b}{17-\i}
    \coordinate (U\i) at ({(\i)*\dx}, 0);
    \fill (\i*\dx, 0) circle (2pt);
    }
    
\foreach \i in {10,...,16} {
    \pgfmathtruncatemacro{\j}{\i + 1}
    \coordinate (mid\i) at ({(\i)*\dx}, -\dy);
    \fill (mid\i) circle (2pt);
   \draw (mid\i) -- (\i*\dx, 0);
   \draw (mid\i) -- (\j*\dx, 0);
    \draw (\i*\dx, 0) -- (\j*\dx, 0);
    \draw (\i*\dx, -\dy) -- (\j*\dx, -\dy);}
\draw (17*\dx,0) -- (17*\dx,-\dy);
\draw (\dx,0) -- (17*\dx,-\dy);
\draw[black, thick]
  ([shift={(0,0)}] 0:2 and 2) 
  arc[start angle=180, end angle=0, x radius=8*\dx, y radius=2];
\draw[black, thick]
  ([shift={(16*\dx,-\dy)}] 0:2 and 2) 
  arc[start angle=0, end angle=-180, x radius=8*\dx, y radius=2];
\end{tikzpicture}
\end{center}
\caption{Improper and proper packing coloring of $T(C_n)$ - situation 4.}
\label{fig:situation4}
\end{figure}
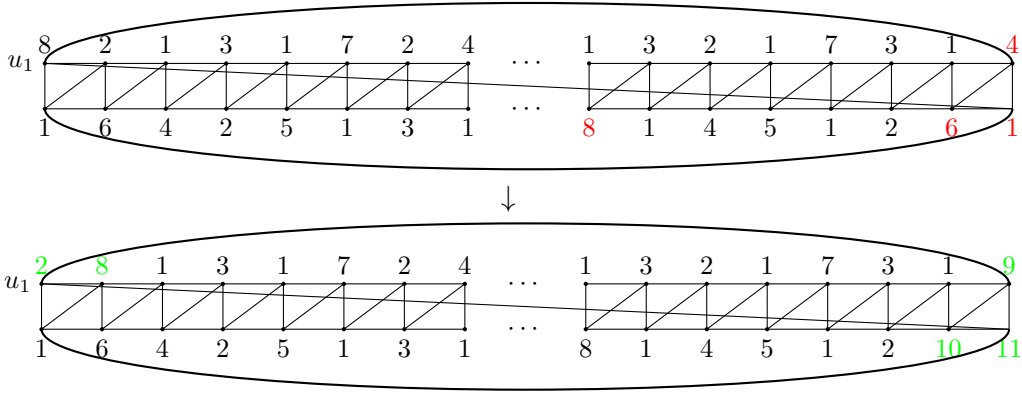
\textbf{Case 2.} Statement (E) holds. \\
Recall that the above presented color pattern for packing coloring of $D(1,2)$ includes number $7$ only three times. From the perspective of a color $7$, a packing coloring $c$ of $T(C_n)$ obtained by applying the mentioned color pattern, is problematic only in the cases when a vertex from $\{u_{n-1},u_{n-1}u_{n},u_n,u_nu_1\}$ receives a color $7$ by $c$. Consequently, we have six problematic situations when $c$ assigns a color $7$ to two distinct vertices of $T(C_n)$, which are at distance at most $7$. Considering the above mentioned color pattern, we observe that for three of the mentioned situations only the statements (E), (F) and either (A) or (B) hold. This means that only three vertices are problematic regarding the current version of a packing coloring of $T(C_n)$. Replacing the colors assigned to these three vertices with three additional colors, namely $9,10,11$, yields an $11$-packing coloring of $T(C_n)$. Hence, in these cases $\chi_\rho^{''}(C_n) \leq 11.$ However, three situations remain in which four problematic colors appear (see 
Fig.~\ref{fig:situation_barva7_1}, Fig.~\ref{fig:situation_barva7_2} and Fig.~\ref{fig:situation_barva7_3}; in each case a graph at the top of the figure). In each of these cases, the presented coloring can be modified to obtain a proper $11$-packing coloring of $T(C_n)$. Indeed, the bottom graphs in Fig.~\ref{fig:situation_barva7_1}, Fig.~\ref{fig:situation_barva7_2} and Fig.~\ref{fig:situation_barva7_3} demonstrate such corrected colorings. These colorings show that, also in these cases, $\chi_\rho^{''}(C_n) \leq 11$, which concludes the proof.

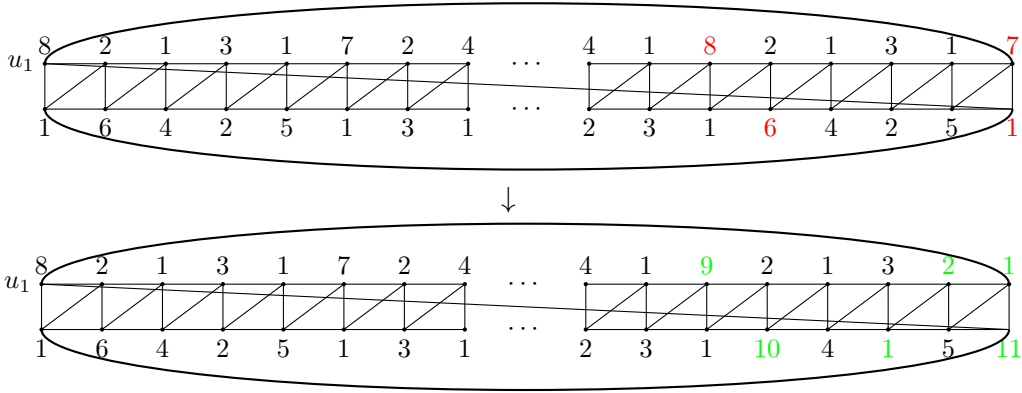
\begin{figure}[h]
\begin{center}
\begin{tikzpicture}[scale=0.4, font=\small]
\def\dx{2} 
\def\dy{1.5} 

\foreach \i in {1,...,8} {
    \fill (\i*\dx, 0) circle (2pt);
  }
\node[left] at (1*\dx, 0) {$u_{1}$};
  \node[above] at (1*\dx, 0) {8};
  \node[above] at (2*\dx, 0) {2};
  \node[above] at (3*\dx, 0) {1};
  \node[above] at (4*\dx, 0) {3};
  \node[above] at (5*\dx, 0) {1};
  \node[above] at (6*\dx, 0) {7};
  \node[above] at (7*\dx, 0) {2};
  \node[above] at (8*\dx, 0) {4};

\foreach \i in {1,...,7} {
    \pgfmathtruncatemacro{\j}{\i + 1}
    \coordinate (mid\i) at ({(\i)*\dx}, -\dy);
    \fill (mid\i) circle (2pt);

    \draw (mid\i) -- (\i*\dx, 0);
    \draw (mid\i) -- (\j*\dx, 0);
    \draw (\i*\dx, 0) -- (\j*\dx, 0);}
    \coordinate (mid8) at ({8*\dx}, -\dy);
    \fill (mid8) circle (2pt);
      \draw (mid8) -- (8*\dx, 0);

\node[below] at (1*\dx, -\dy) {1};
\node[below] at (2*\dx, -\dy) {6};
\node[below] at (3*\dx, -\dy) {4};
\node[below] at (4*\dx, -\dy) {2};
\node[below] at (5*\dx, -\dy) {5};
\node[below] at (6*\dx, -\dy) {1};
\node[below] at (7*\dx, -\dy) {3};
\node[below] at (8*\dx, -\dy) {1};

\foreach \i in {1,...,7} {
    \pgfmathtruncatemacro{\j}{\i + 1}
    \draw (mid\i) -- (mid\j);}
\node at ({(9)*\dx}, 0) {$\cdots$};
\node at ({(9)*\dx}, -\dy) {$\cdots$};
\node[above] at (10*\dx, 0) {4};
\node[above] at (11*\dx, 0) {1};
\node[above] at (12*\dx, 0) {\textcolor{red}{8}};
\node[above] at (13*\dx, 0) {2};
\node[above] at (14*\dx, 0) {1};
\node[above] at (15*\dx, 0) {3};
\node[above] at (16*\dx, 0) {1};
\node[above] at (17*\dx, 0) {\textcolor{red}{7}};

\node[below] at (10*\dx, -\dy) {2};
\node[below] at (11*\dx, -\dy) {3};
\node[below] at (12*\dx, -\dy) {1};
\node[below] at (13*\dx, -\dy) {\textcolor{red}{6}};
\node[below] at (14*\dx, -\dy) {4};
\node[below] at (15*\dx, -\dy) {2};
\node[below] at (16*\dx, -\dy) {5};
\node[below] at (17*\dx, -\dy) {\textcolor{red}{1}};

\foreach \i in {10,...,17} {
 \pgfmathtruncatemacro{\b}{17-\i}
    \coordinate (U\i) at ({(\i)*\dx}, 0);
    \fill (\i*\dx, 0) circle (2pt);
    }


\foreach \i in {10,...,16} {
    \pgfmathtruncatemacro{\j}{\i + 1}
    \coordinate (mid\i) at ({(\i)*\dx}, -\dy);
    \fill (mid\i) circle (2pt);
   \draw (mid\i) -- (\i*\dx, 0);
   \draw (mid\i) -- (\j*\dx, 0);
    \draw (\i*\dx, 0) -- (\j*\dx, 0);
    \draw (\i*\dx, -\dy) -- (\j*\dx, -\dy);}
\draw (17*\dx,0) -- (17*\dx,-\dy);
\draw (\dx,0) -- (17*\dx,-\dy);
\draw[black, thick]
  ([shift={(0,0)}] 0:2 and 2) 
  arc[start angle=180, end angle=0, x radius=8*\dx, y radius=2];
\draw[black, thick]
  ([shift={(16*\dx,-\dy)}] 0:2 and 2) 
  arc[start angle=0, end angle=-180, x radius=8*\dx, y radius=2];
\end{tikzpicture}

$\downarrow$ \vspace{0.1cm}

\begin{tikzpicture}[scale=0.4, font=\small]
\def\dx{2} 
\def\dy{1.5} 
\foreach \i in {1,...,8} {
    \fill (\i*\dx, 0) circle (2pt);}

\node[left] at (1*\dx, 0) {$u_{1}$};
  \node[above] at (1*\dx, 0) {8};
  \node[above] at (2*\dx, 0) {2};
  \node[above] at (3*\dx, 0) {1};
  \node[above] at (4*\dx, 0) {3};
  \node[above] at (5*\dx, 0) {1};
  \node[above] at (6*\dx, 0) {7};
  \node[above] at (7*\dx, 0) {2};
  \node[above] at (8*\dx, 0) {4};
  
\foreach \i in {1,...,7} {
    \pgfmathtruncatemacro{\j}{\i + 1}
    \coordinate (mid\i) at ({(\i)*\dx}, -\dy);
    \fill (mid\i) circle (2pt);
    \draw (mid\i) -- (\i*\dx, 0);
    \draw (mid\i) -- (\j*\dx, 0);
    \draw (\i*\dx, 0) -- (\j*\dx, 0);
}
    \coordinate (mid8) at ({8*\dx}, -\dy);
    \fill (mid8) circle (2pt);
      \draw (mid8) -- (8*\dx, 0);

\node[below] at (1*\dx, -\dy) {1};
\node[below] at (2*\dx, -\dy) {6};
\node[below] at (3*\dx, -\dy) {4};
\node[below] at (4*\dx, -\dy) {2};
\node[below] at (5*\dx, -\dy) {5};
\node[below] at (6*\dx, -\dy) {1};
\node[below] at (7*\dx, -\dy) {3};
\node[below] at (8*\dx, -\dy) {1};

\foreach \i in {1,...,7} {
    \pgfmathtruncatemacro{\j}{\i + 1}
    \draw (mid\i) -- (mid\j);}
\node at ({(9)*\dx}, 0) {$\cdots$};
\node at ({(9)*\dx}, -\dy) {$\cdots$};
\node[above] at (10*\dx, 0) {4};
\node[above] at (11*\dx, 0) {1};
\node[above] at (12*\dx, 0) {\textcolor{green}{9}};
\node[above] at (13*\dx, 0) {2};
\node[above] at (14*\dx, 0) {1};
\node[above] at (15*\dx, 0) {3};
\node[above] at (16*\dx, 0) {\textcolor{green}{2}};
\node[above] at (17*\dx, 0) {\textcolor{green}{1}};

\node[below] at (10*\dx, -\dy) {2};
\node[below] at (11*\dx, -\dy) {3};
\node[below] at (12*\dx, -\dy) {1};
\node[below] at (13*\dx, -\dy) {\textcolor{green}{10}};
\node[below] at (14*\dx, -\dy) {4};
\node[below] at (15*\dx, -\dy) {\textcolor{green}{1}};
\node[below] at (16*\dx, -\dy) {5};
\node[below] at (17*\dx, -\dy) {\textcolor{green}{11}};

\foreach \i in {10,...,17} {
 \pgfmathtruncatemacro{\b}{17-\i}
    \coordinate (U\i) at ({(\i)*\dx}, 0);
    \fill (\i*\dx, 0) circle (2pt);
    }
    
\foreach \i in {10,...,16} {
    \pgfmathtruncatemacro{\j}{\i + 1}
    \coordinate (mid\i) at ({(\i)*\dx}, -\dy);
    \fill (mid\i) circle (2pt);
   \draw (mid\i) -- (\i*\dx, 0);
   \draw (mid\i) -- (\j*\dx, 0);
    \draw (\i*\dx, 0) -- (\j*\dx, 0);
    \draw (\i*\dx, -\dy) -- (\j*\dx, -\dy);}
\draw (17*\dx,0) -- (17*\dx,-\dy);
\draw (\dx,0) -- (17*\dx,-\dy);
\draw[black, thick]
  ([shift={(0,0)}] 0:2 and 2) 
  arc[start angle=180, end angle=0, x radius=8*\dx, y radius=2];
\draw[black, thick]
  ([shift={(16*\dx,-\dy)}] 0:2 and 2) 
  arc[start angle=0, end angle=-180, x radius=8*\dx, y radius=2];
\end{tikzpicture}
\end{center}
\caption{Improper and proper packing coloring of $T(C_n)$ - situation 5.}
\label{fig:situation_barva7_1}
\end{figure}
%
\begin{figure}[h]
\begin{center}
\begin{tikzpicture}[scale=0.4, font=\small]
\def\dx{2} 
\def\dy{1.5} 

\foreach \i in {1,...,8} {
    \fill (\i*\dx, 0) circle (2pt);
  }
\node[left] at (1*\dx, 0) {$u_{1}$};
  \node[above] at (1*\dx, 0) {8};
  \node[above] at (2*\dx, 0) {2};
  \node[above] at (3*\dx, 0) {1};
  \node[above] at (4*\dx, 0) {3};
  \node[above] at (5*\dx, 0) {1};
  \node[above] at (6*\dx, 0) {7};
  \node[above] at (7*\dx, 0) {2};
  \node[above] at (8*\dx, 0) {4};

\foreach \i in {1,...,7} {
    \pgfmathtruncatemacro{\j}{\i + 1}
    \coordinate (mid\i) at ({(\i)*\dx}, -\dy);
    \fill (mid\i) circle (2pt);

    \draw (mid\i) -- (\i*\dx, 0);
    \draw (mid\i) -- (\j*\dx, 0);
    \draw (\i*\dx, 0) -- (\j*\dx, 0);}
    \coordinate (mid8) at ({8*\dx}, -\dy);
    \fill (mid8) circle (2pt);
      \draw (mid8) -- (8*\dx, 0);

\node[below] at (1*\dx, -\dy) {1};
\node[below] at (2*\dx, -\dy) {6};
\node[below] at (3*\dx, -\dy) {4};
\node[below] at (4*\dx, -\dy) {2};
\node[below] at (5*\dx, -\dy) {5};
\node[below] at (6*\dx, -\dy) {1};
\node[below] at (7*\dx, -\dy) {3};
\node[below] at (8*\dx, -\dy) {1};

\foreach \i in {1,...,7} {
    \pgfmathtruncatemacro{\j}{\i + 1}
    \draw (mid\i) -- (mid\j);}
\node at ({(9)*\dx}, 0) {$\cdots$};
\node at ({(9)*\dx}, -\dy) {$\cdots$};
\node[above] at (10*\dx, 0) {4};
\node[above] at (11*\dx, 0) {2};
\node[above] at (12*\dx, 0) {1};
\node[above] at (13*\dx, 0) {5};
\node[above] at (14*\dx, 0) {2};
\node[above] at (15*\dx, 0) {1};
\node[above] at (16*\dx, 0) {\textcolor{red}{6}};
\node[above] at (17*\dx, 0) {\textcolor{red}{2}};

\node[below] at (10*\dx, -\dy) {1};
\node[below] at (11*\dx, -\dy) {3};
\node[below] at (12*\dx, -\dy) {\textcolor{red}{8}};
\node[below] at (13*\dx, -\dy) {1};
\node[below] at (14*\dx, -\dy) {4};
\node[below] at (15*\dx, -\dy) {3};
\node[below] at (16*\dx, -\dy) {1};
\node[below] at (17*\dx, -\dy) {\textcolor{red}{7}};

\foreach \i in {10,...,17} {
 \pgfmathtruncatemacro{\b}{17-\i}
    \coordinate (U\i) at ({(\i)*\dx}, 0);
    \fill (\i*\dx, 0) circle (2pt);
    }


\foreach \i in {10,...,16} {
    \pgfmathtruncatemacro{\j}{\i + 1}
    \coordinate (mid\i) at ({(\i)*\dx}, -\dy);
    \fill (mid\i) circle (2pt);
   \draw (mid\i) -- (\i*\dx, 0);
   \draw (mid\i) -- (\j*\dx, 0);
    \draw (\i*\dx, 0) -- (\j*\dx, 0);
    \draw (\i*\dx, -\dy) -- (\j*\dx, -\dy);}
\draw (17*\dx,0) -- (17*\dx,-\dy);
\draw (\dx,0) -- (17*\dx,-\dy);
\draw[black, thick]
  ([shift={(0,0)}] 0:2 and 2) 
  arc[start angle=180, end angle=0, x radius=8*\dx, y radius=2];
\draw[black, thick]
  ([shift={(16*\dx,-\dy)}] 0:2 and 2) 
  arc[start angle=0, end angle=-180, x radius=8*\dx, y radius=2];
\end{tikzpicture}

$\downarrow$ \vspace{0.1cm}

\begin{tikzpicture}[scale=0.4, font=\small]
\def\dx{2} 
\def\dy{1.5} 
\foreach \i in {1,...,8} {
    \fill (\i*\dx, 0) circle (2pt);}

\node[left] at (1*\dx, 0) {$u_{1}$};
  \node[above] at (1*\dx, 0) {8};
  \node[above] at (2*\dx, 0) {2};
  \node[above] at (3*\dx, 0) {1};
  \node[above] at (4*\dx, 0) {3};
  \node[above] at (5*\dx, 0) {1};
  \node[above] at (6*\dx, 0) {7};
  \node[above] at (7*\dx, 0) {2};
  \node[above] at (8*\dx, 0) {4};
  
\foreach \i in {1,...,7} {
    \pgfmathtruncatemacro{\j}{\i + 1}
    \coordinate (mid\i) at ({(\i)*\dx}, -\dy);
    \fill (mid\i) circle (2pt);
    \draw (mid\i) -- (\i*\dx, 0);
    \draw (mid\i) -- (\j*\dx, 0);
    \draw (\i*\dx, 0) -- (\j*\dx, 0);
}
    \coordinate (mid8) at ({8*\dx}, -\dy);
    \fill (mid8) circle (2pt);
      \draw (mid8) -- (8*\dx, 0);

      \node[below] at (1*\dx, -\dy) {1};
\node[below] at (2*\dx, -\dy) {6};
\node[below] at (3*\dx, -\dy) {4};
\node[below] at (4*\dx, -\dy) {2};
\node[below] at (5*\dx, -\dy) {5};
\node[below] at (6*\dx, -\dy) {1};
\node[below] at (7*\dx, -\dy) {3};
\node[below] at (8*\dx, -\dy) {1};

\foreach \i in {1,...,7} {
    \pgfmathtruncatemacro{\j}{\i + 1}
    \draw (mid\i) -- (mid\j);}
\node at ({(9)*\dx}, 0) {$\cdots$};
\node at ({(9)*\dx}, -\dy) {$\cdots$};
\node[above] at (10*\dx, 0) {4};
\node[above] at (11*\dx, 0) {2};
\node[above] at (12*\dx, 0) {1};
\node[above] at (13*\dx, 0) {5};
\node[above] at (14*\dx, 0) {2};
\node[above] at (15*\dx, 0) {1};
\node[above] at (16*\dx, 0) {\textcolor{green}{9}};
\node[above] at (17*\dx, 0) {\textcolor{green}{1}};

\node[below] at (10*\dx, -\dy) {1};
\node[below] at (11*\dx, -\dy) {3};
\node[below] at (12*\dx, -\dy) {\textcolor{green}{10}};
\node[below] at (13*\dx, -\dy) {1};
\node[below] at (14*\dx, -\dy) {4};
\node[below] at (15*\dx, -\dy) {3};
\node[below] at (16*\dx, -\dy) {\textcolor{green}{2}};
\node[below] at (17*\dx, -\dy) {\textcolor{green}{11}};

\foreach \i in {10,...,17} {
 \pgfmathtruncatemacro{\b}{17-\i}
    \coordinate (U\i) at ({(\i)*\dx}, 0);
    \fill (\i*\dx, 0) circle (2pt);
    }
    
\foreach \i in {10,...,16} {
    \pgfmathtruncatemacro{\j}{\i + 1}
    \coordinate (mid\i) at ({(\i)*\dx}, -\dy);
    \fill (mid\i) circle (2pt);
   \draw (mid\i) -- (\i*\dx, 0);
   \draw (mid\i) -- (\j*\dx, 0);
    \draw (\i*\dx, 0) -- (\j*\dx, 0);
    \draw (\i*\dx, -\dy) -- (\j*\dx, -\dy);}
\draw (17*\dx,0) -- (17*\dx,-\dy);
\draw (\dx,0) -- (17*\dx,-\dy);
\draw[black, thick]
  ([shift={(0,0)}] 0:2 and 2) 
  arc[start angle=180, end angle=0, x radius=8*\dx, y radius=2];
\draw[black, thick]
  ([shift={(16*\dx,-\dy)}] 0:2 and 2) 
  arc[start angle=0, end angle=-180, x radius=8*\dx, y radius=2];
\end{tikzpicture}
\end{center}
\caption{Improper and proper packing coloring of $T(C_n)$ - situation 6.}
\label{fig:situation_barva7_2}
\end{figure}
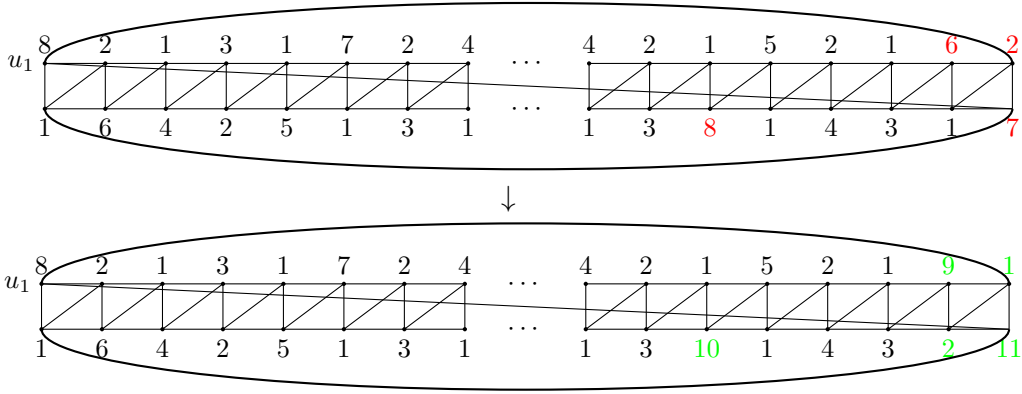
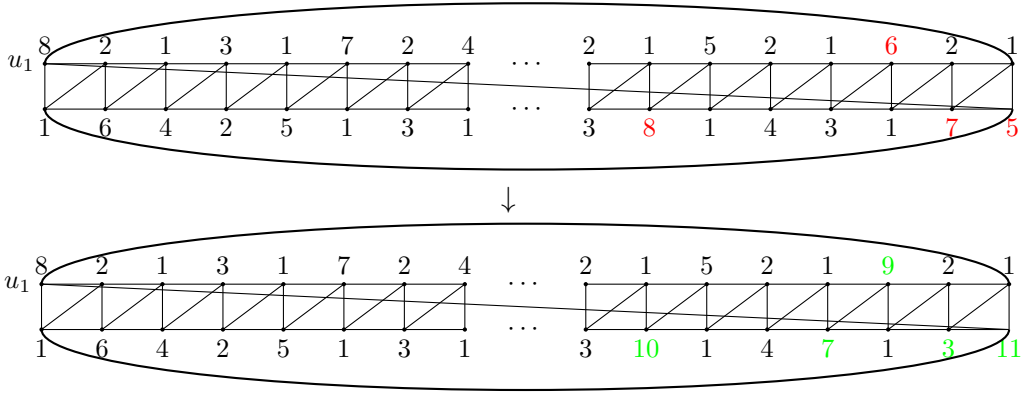
\begin{figure}[h]
\begin{center}
\begin{tikzpicture}[scale=0.4, font=\small]
\def\dx{2} 
\def\dy{1.5} 

\foreach \i in {1,...,8} {
    \fill (\i*\dx, 0) circle (2pt);
  }
\node[left] at (1*\dx, 0) {$u_{1}$};
  \node[above] at (1*\dx, 0) {8};
  \node[above] at (2*\dx, 0) {2};
  \node[above] at (3*\dx, 0) {1};
  \node[above] at (4*\dx, 0) {3};
  \node[above] at (5*\dx, 0) {1};
  \node[above] at (6*\dx, 0) {7};
  \node[above] at (7*\dx, 0) {2};
  \node[above] at (8*\dx, 0) {4};

\foreach \i in {1,...,7} {
    \pgfmathtruncatemacro{\j}{\i + 1}
    \coordinate (mid\i) at ({(\i)*\dx}, -\dy);
    \fill (mid\i) circle (2pt);

    \draw (mid\i) -- (\i*\dx, 0);
    \draw (mid\i) -- (\j*\dx, 0);
    \draw (\i*\dx, 0) -- (\j*\dx, 0);}
    \coordinate (mid8) at ({8*\dx}, -\dy);
    \fill (mid8) circle (2pt);
      \draw (mid8) -- (8*\dx, 0);

\node[below] at (1*\dx, -\dy) {1};
\node[below] at (2*\dx, -\dy) {6};
\node[below] at (3*\dx, -\dy) {4};
\node[below] at (4*\dx, -\dy) {2};
\node[below] at (5*\dx, -\dy) {5};
\node[below] at (6*\dx, -\dy) {1};
\node[below] at (7*\dx, -\dy) {3};
\node[below] at (8*\dx, -\dy) {1};
%
%
\foreach \i in {1,...,7} {
    \pgfmathtruncatemacro{\j}{\i + 1}
    \draw (mid\i) -- (mid\j);}
\node at ({(9)*\dx}, 0) {$\cdots$};
\node at ({(9)*\dx}, -\dy) {$\cdots$};
\node[above] at (10*\dx, 0) {2};
\node[above] at (11*\dx, 0) {1};
\node[above] at (12*\dx, 0) {5};
\node[above] at (13*\dx, 0) {2};
\node[above] at (14*\dx, 0) {1};
\node[above] at (15*\dx, 0) {\textcolor{red}{6}};
\node[above] at (16*\dx, 0) {2};
\node[above] at (17*\dx, 0) {1};

\node[below] at (10*\dx, -\dy) {3};
\node[below] at (11*\dx, -\dy) {\textcolor{red}{8}};
\node[below] at (12*\dx, -\dy) {1};
\node[below] at (13*\dx, -\dy) {4};
\node[below] at (14*\dx, -\dy) {3};
\node[below] at (15*\dx, -\dy) {1};
\node[below] at (16*\dx, -\dy) {\textcolor{red}{7}};
\node[below] at (17*\dx, -\dy) {\textcolor{red}{5}};

\foreach \i in {10,...,17} {
 \pgfmathtruncatemacro{\b}{17-\i}
    \coordinate (U\i) at ({(\i)*\dx}, 0);
    \fill (\i*\dx, 0) circle (2pt);
    }

%
%
\foreach \i in {10,...,16} {
    \pgfmathtruncatemacro{\j}{\i + 1}
    \coordinate (mid\i) at ({(\i)*\dx}, -\dy);
    \fill (mid\i) circle (2pt);
   \draw (mid\i) -- (\i*\dx, 0);
   \draw (mid\i) -- (\j*\dx, 0);
    \draw (\i*\dx, 0) -- (\j*\dx, 0);
    \draw (\i*\dx, -\dy) -- (\j*\dx, -\dy);}
\draw (17*\dx,0) -- (17*\dx,-\dy);
\draw (\dx,0) -- (17*\dx,-\dy);
\draw[black, thick]
  ([shift={(0,0)}] 0:2 and 2) 
  arc[start angle=180, end angle=0, x radius=8*\dx, y radius=2];
\draw[black, thick]
  ([shift={(16*\dx,-\dy)}] 0:2 and 2) 
  arc[start angle=0, end angle=-180, x radius=8*\dx, y radius=2];
\end{tikzpicture}

$\downarrow$ \vspace{0.1cm}

\begin{tikzpicture}[scale=0.4, font=\small]
\def\dx{2} 
\def\dy{1.5} 
\foreach \i in {1,...,8} {
    \fill (\i*\dx, 0) circle (2pt);}

\node[left] at (1*\dx, 0) {$u_{1}$};
  \node[above] at (1*\dx, 0) {8};
  \node[above] at (2*\dx, 0) {2};
  \node[above] at (3*\dx, 0) {1};
  \node[above] at (4*\dx, 0) {3};
  \node[above] at (5*\dx, 0) {1};
  \node[above] at (6*\dx, 0) {7};
  \node[above] at (7*\dx, 0) {2};
  \node[above] at (8*\dx, 0) {4};
  
\foreach \i in {1,...,7} {
    \pgfmathtruncatemacro{\j}{\i + 1}
    \coordinate (mid\i) at ({(\i)*\dx}, -\dy);
    \fill (mid\i) circle (2pt);
    \draw (mid\i) -- (\i*\dx, 0);
    \draw (mid\i) -- (\j*\dx, 0);
    \draw (\i*\dx, 0) -- (\j*\dx, 0);
}
    \coordinate (mid8) at ({8*\dx}, -\dy);
    \fill (mid8) circle (2pt);
      \draw (mid8) -- (8*\dx, 0);

      \node[below] at (1*\dx, -\dy) {1};
\node[below] at (2*\dx, -\dy) {6};
\node[below] at (3*\dx, -\dy) {4};
\node[below] at (4*\dx, -\dy) {2};
\node[below] at (5*\dx, -\dy) {5};
\node[below] at (6*\dx, -\dy) {1};
\node[below] at (7*\dx, -\dy) {3};
\node[below] at (8*\dx, -\dy) {1};

\foreach \i in {1,...,7} {
    \pgfmathtruncatemacro{\j}{\i + 1}
    \draw (mid\i) -- (mid\j);}
\node at ({(9)*\dx}, 0) {$\cdots$};
\node at ({(9)*\dx}, -\dy) {$\cdots$};
\node[above] at (10*\dx, 0) {2};
\node[above] at (11*\dx, 0) {1};
\node[above] at (12*\dx, 0) {5};
\node[above] at (13*\dx, 0) {2};
\node[above] at (14*\dx, 0) {1};
\node[above] at (15*\dx, 0) {\textcolor{green}{9}};
\node[above] at (16*\dx, 0) {2};
\node[above] at (17*\dx, 0) {1};

\node[below] at (10*\dx, -\dy) {3};
\node[below] at (11*\dx, -\dy) {\textcolor{green}{10}};
\node[below] at (12*\dx, -\dy) {1};
\node[below] at (13*\dx, -\dy) {4};
\node[below] at (14*\dx, -\dy) {\textcolor{green}{7}};
\node[below] at (15*\dx, -\dy) {1};
\node[below] at (16*\dx, -\dy) {\textcolor{green}{3}};
\node[below] at (17*\dx, -\dy) {\textcolor{green}{11}};

\foreach \i in {10,...,17} {
 \pgfmathtruncatemacro{\b}{17-\i}
    \coordinate (U\i) at ({(\i)*\dx}, 0);
    \fill (\i*\dx, 0) circle (2pt);
    }
    
\foreach \i in {10,...,16} {
    \pgfmathtruncatemacro{\j}{\i + 1}
    \coordinate (mid\i) at ({(\i)*\dx}, -\dy);
    \fill (mid\i) circle (2pt);
   \draw (mid\i) -- (\i*\dx, 0);
   \draw (mid\i) -- (\j*\dx, 0);
    \draw (\i*\dx, 0) -- (\j*\dx, 0);
    \draw (\i*\dx, -\dy) -- (\j*\dx, -\dy);}
\draw (17*\dx,0) -- (17*\dx,-\dy);
\draw (\dx,0) -- (17*\dx,-\dy);
\draw[black, thick]
  ([shift={(0,0)}] 0:2 and 2) 
  arc[start angle=180, end angle=0, x radius=8*\dx, y radius=2];
\draw[black, thick]
  ([shift={(16*\dx,-\dy)}] 0:2 and 2) 
  arc[start angle=0, end angle=-180, x radius=8*\dx, y radius=2];
\end{tikzpicture}
\end{center}
\caption{Improper and proper packing coloring of $T(C_n)$ - situation 7.}
\label{fig:situation_barva7_3}
\end{figure}
    \qed
\end{proof}

Recall that using a computer we found out that for any $n \geq 14$, $\chi_{\rho}^{''}(P_n) =8$. Since for any $n \geq 14$, a path $P_n$ is a subgraph of $C_n$, we can obtain the following result. 

\begin{remark}
\label{remark:cikli_racunalnik}
For any $n \geq 14$, $\chi_{\rho}^{''}(C_n) \geq 8$.
\end{remark}

With the previous theorem, we have shown that the packing total chromatic numbers of cycles $C_n$ are bounded from above by $11$. However, since for any $n$, $C_n$ is not a subgraph of $C_{n+1}$, the sequence $\chi_\rho^{''}(C_n), \chi_\rho^{''}(C_{n+1})$, $\chi_\rho^{''}(C_{n+2}), \ldots$ is not necessarily non-decreasing. In the following remark, we observe that there are infinitely many cycles with packing total chromatic number equal to $8$, which is particularly interesting since $\chi_\rho^{''}(C_n)>8$ even for small cycles such is $C_7$. 

\begin{remark}
\label{opomba_cikli=8}
    There exist an infinite number of cycles $C_n$, $n \geq 27$, with $\chi_\rho^{''}(C_n) \leq 8$. Indeed, in the cases when $n$ is a multiple of $27$, applying the above presented color pattern for a packing coloring of $T(C_n)$ directly yields a proper packing coloring, which uses only $8$ colors. Additionally, for very large cycles $C_n$, the fact that $\chi_\rho(D(1,2))=8$ implies $\chi_\rho^{''}(C_n) \geq 8$. 
\end{remark}

\section{Connected graphs with small packing total chromatic numbers}

In this section, we characterize all connected graphs $G$ with $\chi_{\rho}^{''}(G) \in \{1,2,3,4,5\}$. Note that we focus only on connected graphs, since in the case when $G$ is not connected, its packing total chromatic number is equal to the maximum value of packing total chromatic numbers of its connected components. 

The only connected graph $G$ with $\chi_{\rho}^{''}(G) = 1$ is a connected graph without edges, namely $K_1$. If a connected graph contains at least two vertices, it contains also at least one edge, which implies that $\chi_{\rho}^{''}(G) \geq 3$. Therefore, there do not exist graphs with packing total chromatic numbers equal to $2$. We also observe that $K_2$ is the only connected graph with the packing total chromatic number equals $3$. Indeed, a connected graph $G$ with at least three vertices contains a subgraph isomorphic to $P_3$ and then by Prop.~\ref{lema_spodnja_meja}, we have $\chi_{\rho}^{''}(G) \geq 4$. Therefore, there are only two connected graphs, $K_1$ and $K_2$, with the packing total chromatic number at most $3$. 

Now, consider the graphs $G$ with $\chi_{\rho}^{''}(G)=4$. Clearly, any such graph $G$ has at least one edge and hence by Prop.~\ref{lema_spodnja_meja}, we know that $\Delta(G) \leq 2.$ Additionally, the previous consideration implies that $\Delta(G) = 2$, which means that $G$ is either a path or a cycle. Then, Theorem~\ref{izrek:poti} and Theorem~\ref{izrek:cikli} imply that the only graph $G$ with $\chi_{\rho}^{''}(G)=4$ is $P_3$. 

Further, the following Proposition shows that there are only four graphs $G$ with $\chi_{\rho}^{''}(G)=5$.

\begin{proposition}
    For any connected graph $G$ holds the following: $\chi_{\rho}^{''}(G)=5$ if and only if $G$ is isomorphic to a graph $P_4, P_5, C_3$ or $K_{1,3}$.
\end{proposition}

\begin{proof}
Let $G$ be a graph isomorphic to a graph from $\{P_4, P_5, C_3, K_{1,3}\}$. Then, by Theorems~\ref{izrek:poti} and~\ref{izrek:cikli}, and the fact that  $\chi_{\rho}^{''}(K_{1,n})=n+2$ for any $n$, we conclude that $\chi_{\rho}^{''}(G)=5$.

Now, suppose that $G$ is a connected graph with $\chi_{\rho}^{''}(G)=5$. In this case, Prop.~\ref{lema_spodnja_meja} yields that $\Delta(G)\leq 3$. Clearly, $\Delta(G) \neq 1$. If $\Delta(G)=2$, then $G$ is a path or a cycle. Moreover, by Theorems~\ref{izrek:poti} and~\ref{izrek:cikli} we know that $G$ is isomorphic to $P_4$, $P_5$ or $C_3$. 

Next, let $\Delta(G) =3$. Denote by $u$ a vertex of $G$ with $deg(u)=3$ and let $N(u)=\{a,b,c\}$. Further, denote by $H$ a subgraph of $G$, which is isomorphic to $K_{1,3}$ and has a vertex set $V(H)=\{u,a,b,c\}$. Note that it is possible that $G$ is isomorphic to $H$ (it is a star $K_{1,3}$), since $\chi_{\rho}^{''}(K_{1,3})=5$. Hence, we can now focus on the case when $G$ is not isomorphic to $H$ and consider any optimal packing total coloring $c$ of $G$. Since $\chi_{\rho}^{''}(H)=5$ and any two elements from $X=V(H) \cup E(H)$ are pairwise at distance at most $2$, for $c$ holds the following: $|c^{-1} (i) \cap X| =1$ for any $i \in \{2,3,4,5\}$ and consequently, $|c^{-1} (1) \cap X| =3$. This implies that either $c(a)=c(b)=c(c)=1$, $c(ua)=c(b)=c(c)=1$, $c(a)=c(ub)=c(c)=1$ or $c(a)=c(b)=c(uc)=1$.  

In the first case there exists $d \in V(G) \setminus V(H)$ which is adjacent to at least one vertices from $\{a,b,c\}$. Without loss of generality we may assume that $ad \in E(G)$. Clearly, $c(ad) \neq 1$ and $c(d) \neq 1$. Further, since the distance between an element from $\{d, ad\}$ and an element from $X$ is at most $3$, we conclude that $c(d)=c(ad)=2$, a contradiction to $c$ being a proper packing total coloring of $G$. 

Now, consider the case when $c(ua)=c(b)=c(c)=1$. It is easy to observe that there $ab, ac, bc \notin E(G)$, since otherwise $c$ cannot be a proper packing total coloring of $G$. Therefore, also in this case there exists $d \in V(G) \setminus V(H)$ which is adjacent to at least one vertices from $\{a,b,c\}$. 
If $ad \in E(G)$, then $c(ad) \neq 1$ and  since the distance between $ad$ and an element from $X \setminus \{b,c\}$ (note that $c(b)=c(c)=1$), is at most $2$, we conclude that $c$ cannot be a proper packing total coloring of $G$. 
If $bd \in E(G)$ (or $cd \in E(G)$), then $c(bd) \neq 1$ and $c(d) \neq 1$. Moreover, since the distance between an element from $\{d, bd\}$ and an element from $X$ is at most $3$, we conclude that $c(d)=c(bd)=2$, again a contradiction. Note that the consideration of the cases when  $c(a)=c(ub)=c(c)=1$ or $c(a)=c(b)=c(uc)=1$ is symmetric to the case when $c(ua)=c(b)=c(c)=1$, hence the proof is done.

\qed
\end{proof}

\section{Open questions and problems}

The introduction of a new concept and the presentation of initial results naturally give rise to several open questions and problems. 

From the definition of a packing total coloring it follows that $\chi_\rho(G) \leq \chi_\rho^{''}(G)$ for any graph $G$. As we have shown, there exist graphs for which the difference between these two invariants can be arbitrarily large. On the other hand, it is not known whether there exists a connected graph $G$ with at least one edge such that $\chi_\rho(G) = \chi_\rho^{''}(G)$. Although at first glance this may seem impossible, we have been unable to prove it. In fact, if such a graph exists, then $\chi_\rho(G) = \chi_\rho^{''}(G) \geq 6$. The absence of a general characterization of graphs $G$ for which $\chi_\rho(G) \geq 5$ does not provide a direct path to a solution.

In this paper, we characterized graphs with small packing total chromatic numbers, namely graphs $G$ satisfying $\chi_\rho^{''}(G) \in \{1,2,3,4,5\}$. It would be interesting to extend these results by providing a characterization of graphs with packing total chromatic number $6$ or more. Moreover, it would be natural to determine the smallest integer $k$ for which the class of graphs $G$ with $\chi_\rho^{''}(G) = k$ is infinite.

Next, we determined the packing total chromatic numbers (or their bounds) only for two very well-known families of graphs: paths and cycles. The study of packing total colorings for other classical graph families remains an open problem. 

We particularly emphasize the problem of determining the exact values of the packing total chromatic numbers for cycles $C_n$, where $n \geq 14$, with respect to the divisibility of $n$. Regarding upper bounds on $\chi_\rho^{''}(C_n)$, recall that Remark~\ref{opomba_cikli=8} says that when $n$ is a multiple of $27$, $\chi_\rho^{''}(C_n) \leq 8$. Furthermore, from the proof of Theorem~\ref{izrek:cikli}, more precisely, from the presented packing total colorings of the cycles $C_{11}$, $C_{12}$ and $C_{13}$, it follows directly that $\chi_\rho^{''}(C_{11k}) \leq 9$, $\chi_\rho^{''}(C_{12k}) \leq 10$ and $\chi_\rho^{''}(C_{13k}) \leq 10$ for all $k \geq 1$.\\ For the lower bound of $\chi_\rho^{''}(C_n)$, where $n \geq 14$, it is in some cases useful to apply a method (similar as in~\cite{togni-2014}) based on determining the maximum possible number of vertices of $T(C_n)$ to which can be assigned a particular color in any optimal packing coloring of $T(C_n)$. More precisely, any optimal packing coloring of $T(C_n)$ assigns a color $i$, $i \geq 1$, to at most $\left\lfloor \frac{2n}{2i+1} \right\rfloor$ vertices. 
For instance, in the case of $C_{14}$, this method directly yields that $\chi_\rho^{''}(C_{14}) \geq 9$. However, note that this does not directly imply that $\chi_\rho^{''}(C_{14k}) \geq 9$ for any $k \geq 1$. We believe that this occurs because, in larger cycles, the sum of the maximum possible numbers of vertices of $T(C_n)$ to which a particular color can be assigned is not precise enough, as this approach does not take into account the fact that each vertex can receive only one color. Hence, making additional considerations regarding the number of vertices of $T(C_n)$ to which two or three selected colors can be assigned in any optimal packing coloring of $T(C_n)$ in some cases leads to further results.
For example, using such additional consideration for colors $1$ and $2$, we can prove that $\chi_\rho^{''}(C_{15}) \geq 9$, but again, this does not directly imply that $\chi_\rho^{''}(C_{15k}) \geq 9$ for any $k \geq 1$. Hence, further systematic research on the described problem is necessary.

Furthermore, the packing total chromatic number of higher powers of paths and cycles seems to be another natural direction for future investigation.

Although the packing total chromatic number $\chi_\rho^{''}(G)$ is unbounded within the class of graphs with maximum degree $\Delta(G) \geq 3$, it is natural to ask whether it becomes bounded when restricted to certain subclasses of subcubic graphs. Moreover, it would be interesting to know whether there exists an upper bound for $\chi_\rho^{''}(G)$, expressed in terms of the maximum degree of $G$, for graphs belonging to some family of subcubic graphs.

We believe that investigating specific subclasses with special structural properties could still yield exact values for packing total chromatic number of some graphs.
One such subclass is that of complete bipartite graphs, where the distance between any two elements of $V(G) \cup E(G)$ is at most 2. As a consequence, any optimal packing total coloring of such graphs uses each color greater than 1 at most once. Therefore, the problem of determining the packing total chromatic number for a complete bipartite graph reduces to finding the largest possible subset $A \subseteq V(G) \cup E(G)$ such that no two vertices in $A$ are adjacent, no two edges in $A$ are incident, and no vertex in $A$ is an endpoint of any edge in $A$. 

It is clear that a graph cannot have a smaller packing total chromatic number than any of its subgraphs; that is, $\chi_\rho^{''}(H) \leq \chi_\rho^{''}(G)$ for every subgraph $H$ of a graph $G$. Graphs $G$ satisfying the property that $\chi_\rho^{''}(H) < \chi_\rho^{''}(G)$ for every proper subgraph $H$ of $G$ are called packing total chromatic critical graphs. These graphs have not yet been studied. For example, it would be interesting to investigate which properties such graphs possess and to identify which graphs belong to this class.

Finally, as mentioned above, packing colorings are a special case of $S$-packing colorings. From this perspective, the concept of packing total coloring coincides with $S$-packing total coloring for the sequence $S = (1, 2, 3, \ldots)$. Therefore, a packing total coloring can be naturally generalized by introducing and studying an $S$-packing total coloring, where $S$ is an arbitrary non-decreasing sequence of positive integers.


\section{Acknowledgements}

J.F. acknowledges the financial support from the Slovenian Research Agency (N1-0431).
We would like to thank the reviewers for their careful reading of the manuscript and for their valuable comments and suggestions, which have helped improve the paper.

%


\end{document}